\documentclass{article}
\usepackage{amssymb}
\usepackage{amsmath}

\newtheorem{theorem}{Theorem}

\newtheorem{corollary}[theorem]{Corollary}

\newtheorem{definition}[theorem]{Definition}

\newtheorem{lemma}[theorem]{Lemma}

\newtheorem{proposition}[theorem]{Proposition}
\newtheorem{remark}[theorem]{Remark}

\newenvironment{proof}[1][Proof]{\noindent\textbf{#1.} }{\ \rule{0.5em}{0.5em}}
\title{\bf Vertical and complete lifts of sections of a (dual) vector bundle and Legendre duality}
\author{E. Peyghan, C. M. Arcu\c{s}, L. Nourmohammadifar}

\begin{document}

\maketitle

\begin{abstract}
Supplementary comments about generalized Lie algebroids are
presented and a new point of view over the construction of the Lie
algebroid generalized tangent bundle of a (dual) vector
bundle is introduced. Using the general theory of exterior
differential calculus for generalized Lie algebroids, a covariant
derivative for exterior forms of a (dual) vector bundle is
introduced. Using this covariant derivative, the complete lift of an
arbitrary section of a (dual) vector bundle is discovered. A theory of Legendre type and Legendre duality between vertical and complete lifts is presented. Finally, a duality between Lie
 algebroids structures is developed.
\end{abstract}

\section{Introduction}
It is well-known that the lift of geometrical objects such as functions,
vector fields and 1-forms defined on the base of the usual Lie
algebroid
\[
( ( TM,\tau _{M},M) ,[ ,] _{TM},( Id_{TM},Id_{M}) ),
\]
has an important role in the geometry of the Lie algebroid%
\[
( ( TTM,\tau _{TM},TM) ,[ ,] _{TTM},( Id_{TTM},Id_{TM}) ).
\]
Using these lifts, it is possible to introduce the lift of a
(pseudo) Riemannian metric structure. The Sasaki lift of a Riemannian metric structure on $M$ is an
important example of metric structures on $TM$ used in differential geometry with many
applications in physics \cite{S}. Lift of geometrical structures of
the Lie algebroid
\[
( ( TM,\tau _{M},M) ,[ ,] _{TM},( Id_{TM},Id_{M}) ),
\]
to the Lie algebroid
\[
( ( TTM,\tau _{TM},TM) ,[ ,] _{TTM},( Id_{TTM},Id_{TM}) ),
\]
were introduced and studied by several authors \cite{OSH, O, YI}.
In many papers such as \cite{EC, PNT, SM, SW}, the
authors studied the lifts to the second order tangent bundle, tensor
bundle and jet bundle.

The Lie algebroids are important issues in physics and mechanics
since the extension of Lagrangian and Hamiltonian systems to their
entity \cite{LMM, M, P} and catching the Poisson structure
\cite{popescu0}. Several authors presented and studied the lift of
geometrical objects of a Lie algebroid $( ( F,\nu ,N) ,[ ,] _{F},(
\rho ,Id_{N}) ) $ to the Lie algebroid prolongation. Using the
vertical and complete lifts of sections of a Lie algebroid, the
first author presented important results about Lie simetry and
horizontal lifts in the general framework of prolongation Lie
algebroid \cite{Pe}.

Extending the notion of Lie algebroid from one base manifold to a
pair of diffeomorphic base manifolds, the second author introduced
the generalized Lie algebroid \cite{A0, A}. Using the lift of a
differentiable curve defined on the base of a generalized Lie
algebroid, he developed a new theory of mechanical systems with many
applications in physics \cite{A1}. The space used for developing
this theory of mechanical systems is the Lie algebroid
generalized tangent bundle%
\[
( ( ( \rho ,\eta ) TF,( \rho ,\eta ) \tau
_{F},F) ,[ ,] _{( \rho ,\eta ) TF},( \tilde{%
\rho},Id_{F}) ),
\]
of a generalized Lie algebroid $( ( F,\nu ,N) ,[ ,] _{F,h},( \rho
,\eta ) )$.

This paper is arranged as follows. Some notions and results about
exterior differential algebra of a vector bundle and information
about generalized Lie algebroids are presented in Section 1. Using a
vector bundle $( E,\pi ,M) $ anchored by a generalized Lie algebroid
$( ( F,\nu ,N) ,[ ,] _{F,F},(
\rho ,\eta ) ) $ and a vector bundle morphism $%
( g,h) $ we obtain a new point of view over construction of the Lie
algebroid generalized tangent bundle
\[
( ( ( \rho ,\eta ) TE,( \rho ,\eta ) \tau
_{E},E) ,[ ,] _{( \rho ,\eta ) TE},( \tilde{%
\rho},Id_{E}) ),
\]%
in Section 2. Using the exterior differential calculus of the
exterior algebra of the generalized Lie algebroid $ ( ( F,\nu ,N) ,[
,] _{F,h},( \rho ,\eta ) ) $ presented in \cite{A00}, in Section 3,
we introduce a Lie covariant derivative for the exterior algebra of
the vector bundle $( E,\pi
,M) .$ Using this Lie covarinat derivative, we introduce in \emph{%
Theorem 11} the complete $(g, h)$-lift $u^{c}\in \Gamma ( TE,\tau
_{E},E) $ of an arbitrary section $u\in \Gamma ( E,\pi ,M)$. Using
the complete $( g,h) $-lift of a function $f\in
\mathcal{F}(N)$ we obtain new results for vertical and complete
$(g,h)$-lifts. In the final of Section 3, we introduced the complete
and vertical $(g, h)$-lifts
\[
u^{C},u^{V}\in \Gamma ( ( \rho ,\eta ) TE,( \rho ,\eta ) \tau
_{E},E),
\]
of a section $u\in \Gamma ( E,\pi ,M) $ and important
results are presented in \emph{Theorem 18}. Also, using the dual vector bundle $( \overset{\ast }{E},\overset{\ast }{\pi }%
,M) $ anchored by a generalized Lie algebroid $( ( F,\nu ,N) ,[ ,]
_{F,h},( \rho ,\eta ) ) $ and a vector bundle morphism
$( \overset{\ast }{g},h)$, in Section 4, we obtain a new point of
view over construction of the Lie algebroid generalized tangent
bundle
\[
( ( ( \rho ,\eta ) T\overset{\ast }{E},( \rho
,\eta ) \tau _{\overset{\ast }{E}},\overset{\ast }{E}) ,[ ,%
] _{( \rho ,\eta ) T\overset{\ast }{E}},( \overset{%
\ast }{\tilde{\rho}},Id_{\overset{\ast }{E}}) ).
\]
A dual theory for the vertical and complete lifts is presented in
Section 5 and similar results are obtained. A general presentation of Lagrange (Finsler) and Hamilton
(Cartan) fundamental functions and a theory of Legendre type are
presented in Section 6. Using the tangent $\left( \rho ,\eta \right)
$-application of the Legendre bundle morphism associated to a
Lagrange respectively Hamilton fundamental function, we obtain new
results about duality between vertical and complete $(g, h)$-lifts in Section 7. New results about duality between Lie
algebroids structures and the Legendre $( \rho ,\eta) $%
-equivalence between the vector bundle $( E,\pi ,M) $ and its dual
$( \overset{\ast }{E},\overset{\ast }{\pi },M) $ are presented in
Section 8.
\section{Preliminaries}
Let $(E,\pi ,M)$ be an arbitrary vector bundle. If  $\Gamma (E,\pi ,M)$ is the set of the sections of the vector bundle $%
(E,\pi ,M)$ and $\mathcal{F}(M)$ is the set of differentiable
real-valued functions on $M$, then $(\Gamma (E,\pi ,M),+,\cdot )$ is
a $\mathcal{F}(M)$-module.

For any $q\in \mathbb{N}$ we denote by $( \Sigma _{q},\circ ) $ the
permutations group of the set $\{ 1,2,...,q\}$. We denoted by
$\Lambda ^{q}(E,\pi ,M)$ the set of $q$-linear applications
\begin{equation*}
\begin{array}{ccc}
\Gamma (E,\pi ,M)^{q} & ^{\underrightarrow{\ \ \omega \ \ }} & \mathcal{F}%
( M),\\
( z_{1},\cdot,z_{q}) & \longmapsto & \omega ( z_{1},\cdot,z_{q}),
\end{array}%
\end{equation*}%
such that
\begin{equation*}
\begin{array}{c}
\omega ( z_{\sigma ( 1) },\cdot,z_{\sigma ( q) }) =sgn( \sigma )
\cdot \omega ( z_{1},\cdot,z_{q}),
\end{array}%
\end{equation*}%
for any $z_{1},...,z_{q}\in \Gamma (E,\pi ,M)$ and for any $\sigma
\in \Sigma _{q}$. The elements of $\Lambda ^{q}(E,\pi ,M)$ will be
called \emph{differential forms of degree }$q$ or \emph{differential
}$q$\emph{-forms}. It is known that $( \Lambda ^{q}(E,\pi ,M),+,\cdot ) $\ is a $\mathcal{F}%
( M) $-module \cite{A00}.

\begin{definition}
If $\omega \in \Lambda ^{q}(E,\pi ,M)$ and $\theta \in \Lambda
^{r}(E,\pi ,M) $, then the $( q+r) $-form $\omega \wedge \theta $
defined by
\begin{align*}
\omega \wedge \theta ( u_{1},...,u_{q+r})&=\underset{\underset{\sigma ( q+1) <...<\sigma ( q+r)}{\sigma ( 1) <...<\sigma ( q)}}{\sum}{sgn(
\sigma )} \omega ( u_{\sigma ( 1) },...,u_{\sigma
( q) }) \theta ( u_{\sigma ( q+1)
},...,u_{\sigma ( q+r) })\nonumber\\
&=\frac{1}{q!r!}\underset{\sigma \in \Sigma _{q+r}}{\sum}
sgn( \sigma ) \omega ( u_{\sigma ( 1)
},...,u_{\sigma ( q) }) \theta ( u_{\sigma (
q+1) },...,u_{\sigma ( q+r) }),
\end{align*}
for any $u_{1},...,u_{q+r}\in \Gamma (E,\pi ,M),$ will be called
{the exterior product of the forms }$\omega $ \emph{and}~$\theta .$
\end{definition}
Using the previous definition, we obtain
\begin{theorem}
Let $\omega, \sigma \in \Lambda ^{q}(E,\pi ,M)$, $%
\theta \in \Lambda ^{r}(E,\pi ,M)$, $\eta, \xi\in \Lambda ^{s}(E,\pi
,M)$ and $f\in\mathcal{F}(M)$. Then
\begin{align*}
&\omega \wedge \theta=( -1) ^{qr}\theta \wedge \omega,\ \ \ \ (
\omega \wedge \theta ) \wedge \eta =\omega
\wedge ( \theta \wedge \eta ),\\
&( \omega +\sigma ) \wedge \eta=\omega \wedge \eta
+\sigma\wedge\eta,\ \ \omega \wedge (\eta+\xi ) =\omega
\wedge \eta+\omega \wedge \xi,\\
&( f\omega ) \wedge \theta =f( \omega \wedge \theta )=\omega \wedge
( f\theta ) .
\end{align*}
\end{theorem}
We set
\begin{equation*}
\Lambda (E,\pi ,M):=\underset{q\geq 0}{\oplus }\Lambda ^{q}(E,\pi
,M).
\end{equation*}
Then it is easy to see that $( \Lambda (E,\pi ,M),+,\cdot ,\wedge )
$ is a $\mathcal{F}( M) $-algebra. This algebra will be called
\emph{the exterior differential algebra of the vector bundle
}$(E,\pi ,M)$.

Now let $(F,\nu ,N)$ be an another vector bundle and $(\varphi
,\varphi _{0})$ is
a vector bundles morphism from $(E,\pi ,M)$ to $(F,\nu ,N)$ such that $%
\varphi _{0}$ is a isomorphism from $M$ to $N$. Then using the
operation
\begin{equation*}
\begin{array}{ccc}
\mathcal{F}( M) \times \Gamma ( F,\nu ,N) & ^{%
\underrightarrow{~\ \ \cdot ~\ \ }} & \Gamma ( F,\nu ,N) , \\
( f,z) & \longmapsto & f\circ \varphi _{0}^{-1}\cdot z,%
\end{array}%
\end{equation*}%
it results that $(\Gamma (F,\nu ,N),+,\cdot )$ is a $\mathcal{F}(M)$-module
and we obtain the modules morphism
\begin{equation*}
\begin{array}{ccc}
\Gamma (E,\pi ,M) & ^{\underrightarrow{~\ \ \Gamma \left( \varphi ,\varphi
_{0}\right) ~\ \ }} & \Gamma \left( F,\nu ,N\right) , \\
u=u^{a}s_{a} & \longmapsto & \Gamma \left( \varphi ,\varphi _{0}\right) u,%
\end{array}%
\end{equation*}%
defined by
\begin{equation*}\label{*}
\begin{array}{c}
\Gamma \left( \varphi ,\varphi _{0}\right) u=\left( u^{a}\circ \varphi
_{0}^{-1}\right) \left( \varphi _{a}^{\alpha }\circ \varphi _{0}^{-1}\right)
t_{\alpha }.%
\end{array}%
\end{equation*}
\begin{definition}
Let $(\varphi ,\varphi _{0})$ be a vector bundles morphism from
$(E,\pi ,M)$ to $(F,\nu ,N)$ such that $\varphi _{0}$ is a
isomorphism from $M$ to $N$. Then we define the pull-back
application
\begin{equation*}
\begin{array}{ccc}
\Lambda ^{q}\left( F,\nu ,N\right) & ^{\underrightarrow{\ \left( \varphi
,\varphi _{0}\right) ^{\ast }\ }} & \Lambda ^{q}(E,\pi ,M),\\
\omega & \longmapsto & \left( \varphi ,\varphi _{0}\right) ^{\ast }\omega,%
\end{array}%
\end{equation*}%
by
\begin{equation*}
\begin{array}{c}
\left( \left( \varphi ,\varphi _{0}\right) ^{\ast }\omega \right) \left(
u_{1},...,u_{q}\right) =\omega \left( \Gamma \left( \varphi ,\varphi
_{0}\right) \left( u_{1}\right) ,...,\Gamma \left( \varphi ,\varphi
_{0}\right) \left( u_{q}\right) \right) ,%
\end{array}%
\end{equation*}%
for any $u_{1},...,u_{q}\in \Gamma (E,\pi ,M).$ If $f\in
\mathcal{F}( N) ,$ then $( \varphi ,\varphi _{0}) ^{\ast }( f)
=f\circ \varphi _{0}.$
\end{definition}
\begin{remark}
If $(\rho ,\eta )$ and $( Th,h) $ are two vector bundles
morphisms given by the diagrams%
\begin{equation*}
\begin{array}[b]{ccccc}
F & ^{\underrightarrow{~\ \ \ \rho \ \ \ \ }} & TM & ^{\underrightarrow{~\ \
\ Th\ \ \ \ }} & TN \\
~\downarrow \nu  &  & ~\ \ \downarrow \tau _{M} &  & ~\ \ \downarrow \tau
_{N} \\
N & ^{\underrightarrow{~\ \ \ \eta ~\ \ }} & M &
^{\underrightarrow{~\ \ \ h~\ \ }} & N
\end{array}
,
\end{equation*}
then we obtain the modules morphism $\Gamma ( Th\circ \rho
,h\circ \eta ) $ different by $\Gamma ( Th,h) \circ \Gamma ( \rho
,\eta ) $ defined by
\begin{equation*}
\begin{array}{c}
\Gamma ( Th\circ \rho ,h\circ \eta ) ( z^{\alpha }t_{\alpha
}) ( f) =z^{\alpha }\rho _{\alpha }^{i}( \frac{%
\partial ( f\circ h) }{\partial x^{i}}) \circ h^{-1},%
\end{array}%
\end{equation*}%
for any $z^{\alpha }t_{\alpha }\in \Gamma ( F,\nu ,N) $ and $f\in
\mathcal{F}( N) .$
\end{remark}
\begin{definition}
A generalized Lie algebroid is a vector bundle $(F,\nu ,N)$ given by the
diagrams:
\begin{equation*}
\begin{array}{c}
\begin{array}[b]{ccccc}
( F,[ ,] _{F,h}) & ^{\underrightarrow{~\ \ \ \rho \ \ \ \ }} & (
TM,[ ,] _{TM}) & ^{\underrightarrow{~\ \ \
Th\ \ \ \ }} & ( TN,[ ,] _{TN}) \\
~\downarrow \nu &  & ~\ \ \downarrow \tau _{M} &  & ~\ \ \downarrow \tau _{N}
\ \ \ \ \ \ \ ,\\
N & ^{\underrightarrow{~\ \ \ \eta ~\ \ }} & M & ^{\underrightarrow{~\ \ \
h~\ \ }} & N%
\end{array}%
\end{array}%
\end{equation*}%
where $h$ and $\eta $ are arbitrary isomorphisms, $(\rho ,\eta )$ is a
vector bundles morphism from $(F,\nu ,N)$ to $(TM,\tau _{M},M)$ and the
operation
\begin{equation*}
\begin{array}{ccc}
\Gamma ( F,\nu ,N) \times \Gamma ( F,\nu ,N) & ^{%
\underrightarrow{~\ \ [ ,] _{F,h}~\ \ }} & \Gamma ( F,\nu
,N) , \\
( u,v) & \longmapsto & \ [ u,v] _{F,h},%
\end{array}%
\end{equation*}%
satisfies in
\begin{equation*}
\begin{array}{c}
[ u,f\cdot v] _{F,h}=f[ u,v] _{F,h}+\Gamma ( Th\circ \rho ,h\circ
\eta ) ( u) f\cdot v,\ \ \ \forall
f\in \mathcal{F}(N),%
\end{array}%
\end{equation*}%
such that the 4-tuple $(\Gamma (F,\nu ,N),+,\cdot ,[,]_{F,h})$ is a Lie $%
\mathcal{F}(N)$-algebra.
\end{definition}
We denote by $((F, \nu, N), [ , ] _{F,h}, (\rho, \eta))$ the
generalized Lie algebroid defined in the above. Moreover, the couple $(%
\lbrack\ , ] _{F,h}, (\rho, \eta))$ is called the
\emph{generalized Lie algebroid structure.} In particular, if $\eta
=Id_{M}=h,$ then we obtain the definition of Lie algebroid. So, any
Lie algebroid can be regarded as a generalized Lie algebroid.
\begin{remark}\label{R1}
The modules morphism $\Gamma \left( Th\circ \rho ,h\circ \eta
\right) $ is
Lie algebras morphism from $(\Gamma (F,\nu ,N),+,\cdot ,[,]_{F,h})$ to $%
(\Gamma (TN,\tau _{N},N),+,\cdot ,[,]_{TN})$.
\end{remark}
If we take local coordinates $(x^{i})$ and $(\chi
^{\tilde{\imath}})$ on open sets $V\subset M$ and $W\subset N$,
respectively, then we have the
corresponding local coordinates $(x^{i},y^{i})$ and $(\chi ^{\tilde{\imath}%
},z^{\tilde{\imath}})$ on $TM$ and $TN$, respectively, where $i,\tilde{\imath%
}\in \{1,\ldots ,m\}$. Moreover, we consider $(\chi ^{\tilde{%
\imath}},z^{\alpha })$ as the local coordinates on $F$, where $\alpha \in \{1,\ldots ,p\}$.
If $\{t_\alpha\}$ is a local basis for module of sections of $(F,\nu,N)$, then we put $[t_{\alpha },t_{\beta }]_{F,h}=L_{\alpha \beta }^{\gamma
}t_{\gamma }$, where $L_{\alpha \beta }^{\gamma }$ are local
functions on $N$
and $\alpha ,\beta ,\gamma \in \{1,\ldots ,p\}$. It is easy to see that $%
L_{\alpha \beta }^{\gamma }=-L_{\beta \alpha }^{\gamma }$. Using
Remark \ref{R1}, we obtain
\begin{equation}
\begin{array}{c}
\displaystyle( L_{\alpha \beta }^{\gamma }\circ h)( \rho
_{\gamma }^{k}\circ h) =( \rho _{\alpha }^{i}\circ h) \frac{%
\partial( \rho _{\beta }^{k}\circ h) }{\partial x^{i}}-(
\rho _{\beta }^{j}\circ h) \frac{\partial( \rho _{\alpha
}^{k}\circ h) }{\partial x^{j}}.%
\end{array}
\label{P1}
\end{equation}%
The local functions $L_{\alpha \beta }^{\gamma }$ introduced in the
above are called the {\textit{structure functions}} of the
generalized Lie algebroid $((F,\nu ,N),[,]_{F,h},(\rho ,\eta
))$.

A morphism from $(( F,\nu ,N), [ , ] _{F,h}, (\rho, \eta)) $
to $(( F^{\prime }, \nu ^{\prime }, N^{\prime }), [ , ]
_{F^{\prime }, h^{\prime }}, (\rho ^{\prime }, \eta ^{\prime }))
$ is a morphism $(\varphi ,\varphi _{0})$ from $(F,\nu, N)$ to
$(F^{\prime },\nu^{\prime }, N^{\prime })$ such that $\varphi _{0}$
is an isomorphism from $N$ to $N^{\prime }$, and the modules
morphism $\Gamma(\varphi, \varphi _{0})$ is a Lie algebras morphism
from $( \Gamma( F,\nu ,N) ,+,\cdot, [ ,]
_{F,h}) $ to $( \Gamma ( F^{\prime },\nu ^{\prime
},N^{\prime }) ,+,\cdot , [ ,] _{F^{\prime
},h^{\prime }}). $ Thus, we can discuss about the category of
generalized Lie algebroids.

\subsection{The generalized tangent bundle of a vector
bundle}
We consider the following diagrams:
\begin{equation}\label{diag}
\begin{array}{ccccccc}
E & ^{\underrightarrow{~\ \ g~\ \ }} & ( F,[ ,] _{F,h})
& ^{\underrightarrow{~\ \ \rho ~\ \ }} & TM & ^{\underrightarrow{~\ \ Th~\ \
}} & ( TN,[ ,] _{TN}) \\
~\downarrow \pi &  & \downarrow \nu &  & ~\ \ \downarrow \tau _{M} &  & ~\ \
\downarrow \tau _{N}\ \ \ \ \ \ \ \ \ \ , \\
M & ^{\underrightarrow{~\ \ h~\ \ }} & N & ^{\underrightarrow{~\ \ \eta ~\ \
}} & M & ^{\underrightarrow{~\ \ h~\ \ }} & N%
\end{array}%
\end{equation}%
where $(( F,\nu ,N) ,[ ,] _{F,h},( \rho ,\eta ) )$ is a generalized
Lie algebroid, $( E,\pi ,M) $ is a vector bundle and $( g,h) $ is a
vector bundle morphism from $( E,\pi ,M) $ to $( F,\nu ,N) $ with
components $g_{b}^{\alpha },~\alpha \in \{ 1,2,\cdots,n\} $ and
$b\in \{ 1,2,\cdots,r\} $.

If $v=v^{\alpha }t_{\alpha }$ is a section of $(F,\nu ,N)$, then
we define its corresponding section $X=X^{\alpha }T_{\alpha }$ in
the pull-back vector bundle $( ( h\circ \pi ) ^{\ast }F,( h\circ
\pi ) ^{\ast }\nu ,{E}) $ given by
\begin{align*}
X( u_{x})=(( g_{b}^{\alpha }\circ h^{-1})
( u^{b}\circ h^{-1}) t_{\alpha })( h\circ \pi(
u_{x}))=g_{b}^{\alpha }( x) u^{b}( x) t_{\alpha }(h(x)),
\end{align*}%
for any $u_{x}\in E$. If we define
\begin{equation*}
\begin{array}{rcl}
\ ( h\circ \pi ) ^{\ast }F & ^{\underrightarrow{\overset{(
h\circ \pi ) ^{\ast }F}{\rho }}} & TE, \\
\displaystyle X^{\alpha }T_{\alpha }(u_{x}) & \longmapsto & \displaystyle%
( g_{b}^{\alpha }\circ \pi ) (u^{b}\circ \pi )( \rho
_{\alpha }^{i}\circ h\circ \pi ) \partial _{i}( u_{x}) ,%
\end{array}%
\end{equation*}%
then $(({\overset{( h\circ \pi ) ^{\ast }E}{\rho }},Id_{E}))$ is
a vector bundles morphism from $( ( h\circ \pi ) ^{\ast }F,(
h\circ \pi ) ^{\ast }\nu ,E) $ to $( TE,\tau _{E},E) $. Moreover,
the operation
\begin{equation*}
\begin{array}{ccc}
\Gamma ( ( h\circ \pi ) ^{\ast }F,( h\circ \pi ) ^{\ast }\nu ,E)
^{2} & ^{\underrightarrow{~\ \ [ ,] _{( h\circ \pi ) ^{\ast }F}~\
\ }} & \Gamma ( ( h\circ
\pi ) ^{\ast }F,( h\circ \pi ) ^{\ast }\nu ,E) ,%
\end{array}%
\end{equation*}%
defined by
\begin{equation*}
\begin{array}{ll}
\left[ T_{\alpha },T_{\beta }\right] _{\left( h\circ \pi \right) ^{\ast }F}
&\!\!\!\!\!\!=(L_{\alpha \beta }^{\gamma }\circ h\circ \pi )T_{\gamma },\vspace*{1mm}
\\
\left[ T_{\alpha },fT_{\beta }\right] _{\left( h\circ \pi \right) ^{\ast }F}
&\!\!\!=f( L_{\alpha \beta }^{\gamma }\circ h\circ \pi)
T_{\gamma }+(\rho _{\alpha }^{i}\circ h\circ \pi) \partial
_{i}\left( f\right) T_{\beta },\vspace*{1mm} \\
\left[ fT_{\alpha },T_{\beta }\right] _{\left( h\circ \pi \right) ^{\ast }F}
&\!\!\!=-\left[ T_{\beta },fT_{\alpha }\right] _{\left( h\circ \pi \right) ^{\ast
}F},%
\end{array}%
\end{equation*}%
for any $f\in \mathcal{F}( E) $, is a Lie bracket on $\Gamma
(( h\circ \pi) ^{\ast }F,( h\circ \pi) ^{\ast
}\nu ,E) $. It is easy to check that
\begin{equation*}
((\left( h\circ \pi \right) ^{\ast }F,\left( h\circ \pi \right)
^{\ast
}\nu ,E),\left[ ,\right] _{\left( h\circ \pi \right) ^{\ast }F},(\overset{%
\left( h\circ \pi \right) ^{\ast }F}{\rho },Id_{E})),
\end{equation*}%
is a Lie algebroid which is called the \textit{pull-back Lie algebroid of
the generalized Lie algebroid} $(\left( F,\nu ,N\right) ,\left[ ,\right]
_{F,h},\left( \rho ,\eta \right) )$.

Let $(\partial _{i},\dot{\partial}_{a})$ be the base sections for the Lie $%
\mathcal{F}(E)$-algebra
\begin{equation*}
(\Gamma (TE,\tau _{E},E),+,\cdot ,[,]_{TE}).
\end{equation*}%
Then
\begin{align*}
X^{\alpha }\tilde{\partial}_{\alpha }+\tilde{X}^{a}\dot{\tilde{\partial}}%
_{a}& :=X^{\alpha }(T_{\alpha }\oplus (\rho _{a}^{i}\circ h\circ \pi
)\partial _{i})+\tilde{X}^{a}(0_{\left( h\circ \pi \right) ^{\ast }F}\oplus
\dot{\partial}_{a}) \\
& =X^{\alpha }T_{\alpha }\oplus (X^{\alpha }(\rho _{\alpha }^{i}\circ h\circ
\pi )\partial _{i}+\tilde{X}^{a}\dot{\partial}_{a}),
\end{align*}%
making from $X^{\alpha }S_{\alpha }\in \Gamma (( h\circ \pi
) ^{\ast }F,( h\circ \pi) ^{\ast }\nu ,E) $ and $%
\tilde{X}^{a}\dot{\partial}_{a}\in \Gamma( VTE,\tau _{E},E) $ is
a section of $(( h\circ \pi) ^{\ast }F\oplus TE,\overset{\oplus }%
{\pi },E)$. Moreover, it is easy to see that the sections $\tilde{\partial}%
_{1},\cdots ,\tilde{\partial}_{r},\dot{\tilde{\partial}}_{1},\cdots ,\dot{%
\tilde{\partial}}_{r}$ are linearly independent. Now, we consider the vector
subbundle $(\left( \rho ,\eta \right) TE,\left( \rho ,\eta \right) \tau
_{E},E)$ of the vector bundle $(\left( h\circ \pi \right) ^{\ast }F\oplus TE,%
\overset{\oplus }{\pi },E),$ for which the $\mathcal{F}(E)$-module of
sections is the $\mathcal{F}(E)$-submodule of
\begin{equation*}
(\Gamma (\left( h\circ \pi \right) ^{\ast }F\oplus TE,\overset{\oplus }{\pi }%
,E),+,\cdot ),
\end{equation*}%
generated by the set of sections $(\tilde{\partial}_{\alpha },\dot{\tilde{%
\partial}}_{a}).$ The base sections $(\tilde{\partial}_{\alpha },\dot{\tilde{%
\partial}}_{a})$ are called the \emph{natural }$(\rho ,\eta )$\emph{-base}.
Now, we consider the vector bundles morphism $\left( \tilde{\rho}%
,Id_{E}\right) $ from $((\rho ,\eta )TE,(\rho ,\eta )\tau _{E}$ $,E)$ to $%
(TE,\tau _{E},E)$ , where
\begin{equation*}
\begin{array}{rcl}
\left( \rho ,\eta \right) TE\!\!\! & \!\!^{\underrightarrow{\ \ \tilde{\rho}%
\ \ }}\!\!\! & \!\!TE\vspace*{2mm} ,\\
(X^{\alpha }\tilde{\partial}_{\alpha }+\tilde{X}^{a}\dot{\tilde{\partial}}%
_{a})\!(u_{x})\!\!\!\! & \!\!\longmapsto \!\!\! & \!\!(\!X^{\alpha }(\rho
_{\alpha }^{i}{\circ h\circ }\pi )\partial _{i}{+}\tilde{X}^{a}\dot{\partial}%
_{a})\!(u_{x}).%
\end{array}%
\end{equation*}%
Moreover, define the Lie bracket $[,]_{(\rho ,\eta )TE}$ as follows%
\begin{equation*}
\begin{array}{l}
[X_{1}^{\alpha }\tilde{\partial}_{\alpha }+\tilde{X}_{1}^{a}\dot{\tilde{%
\partial}}_{a},X_{2}^{\beta }\tilde{\partial}_{\beta }+\tilde{X}_{2}^{b}\dot{%
\tilde{\partial}}_{b}]_{( \rho ,\eta ) TE} =[X_{1}^{\alpha
}T_{\alpha }, X_{2}^{\beta }T_{\beta }]_{(
h\circ \pi ) ^{\ast }F} \\
\ \ \ \oplus[( X_{1}^{\alpha }(\rho _{\alpha }^{i}\circ h\circ \pi
)\partial _{i}+\tilde{X}_{1}^{a}\dot{\partial}_{a},X_{2}^{\beta
}(\rho _{\beta
}^{j}\circ h\circ \pi )\partial _{j}+\tilde{X}_{2}^{b}\dot{\partial}%
_{b})] _{TE}.%
\end{array}%
\end{equation*}%
Easily we obtain that $([ ,] _{( \rho ,\eta ) TE},(
\tilde{\rho},Id_{E}) )$ is a Lie algebroid structure for the vector
bundle $( ( \rho ,\eta ) TE,( \rho ,\eta ) \tau _{E},E) $ which is
called the \textit{generalized tangent bundle}.
\section{Vertical and complete $( g,h) $-lifts of sections of a vector bundle}
 In this section, we consider the diagram (\ref{diag}) for the generalized Lie algebroid
$( ( F,\nu ,N) ,[ ,] _{F,h},( \rho ,\eta ) ) $. Also, we admit that
$( g,h) $ is a vector bundles morphism locally invertible from $(
E,\pi ,M) $ to $( F,\nu ,N) $ with components
\begin{equation*}
\begin{array}{c}
g_{b}^{\alpha },\ \ \alpha \in \{ 1,\cdots ,n\},\ \  b\in \{
1,\ldots ,r\} .%
\end{array}%
\end{equation*}
So, for any vector local $( m+r) $-chart $( V,t_{V}) $ of $( E,\pi
,M) $, there exist the real functions
\begin{equation*}
\begin{array}{ccc}
V & ^{\underrightarrow{~\ \ \ \tilde{g}_{\alpha }^{b}~\ \ }}\!\!\!\! &
\mathbb{R},\ \ \alpha \in \{1,\cdots,n\},\ \  b\in \{1,\cdots,r\},%
\end{array}%
\end{equation*}%
such that $\tilde{g}_{\alpha }^{b}( \varkappa ) \cdot g_{a}^{\alpha
}( \varkappa ) =\delta _{a}^{b}$ and $\tilde{g}_{\alpha }^{a}(
\varkappa ) \cdot g_{a}^{\beta }( \varkappa ) =\delta _{\alpha
}^{\beta }$, for any $\varkappa \in V$.

Thus, we can discuss about vector bundles morphism $\left(
g^{-1},h^{-1}\right) $ from $\left( F,\nu ,N\right) $ to $\left(
E,\pi ,M\right) $ with components
\begin{equation*}
\begin{array}{c}
\tilde{g}_{\alpha }^{b}\circ h^{-1},\ \alpha \in \{1,\cdots,n\},\ \
\ \
b\in \{1,\cdots,r\}.%
\end{array}%
\end{equation*}

\begin{definition}
If $f\in \mathcal{F}\left( N\right) $ (respectively $f\in
\mathcal{F}\left( M\right))$, then the real function $f^{\vee
}=f\circ h\circ \pi $ (respectively $f^{\vee }=f\circ \pi )$ is
called the vertical lift of the function$f.$
\end{definition}
It is remarkable that since
\begin{equation*}
\begin{array}{c}
( \Gamma (Th\circ \rho ,h\circ \eta )\Gamma ( g,h) (u^{a}s_{a}))
(f)=( g_{b}^{\alpha }u^{b}) \circ h^{-1}\rho
_{\alpha }^{i}\frac{\partial (f\circ h)}{\partial x^{i}}\circ h^{-1}%
\end{array}%
,
\end{equation*}%
then using the above definition
\begin{equation}\label{J1}
( \Gamma (Th\circ \rho,h\circ \eta)\Gamma(g, h)(u^{a}s_{a})(f)%
) ^{v}=( ( g_{b}^{\alpha }u^{b}\rho _{\alpha }^{i}\circ h) \circ \pi
) \partial _{i}(f\circ h\circ \pi ).
\end{equation}
\begin{definition}
If $u=u^{a}s_{a}$ is a section of $( E,\pi ,M) $, then we introduce
the vertical lift of $u$ as section of $\Gamma ( TE,\tau _{E},E) $
given by
\begin{equation}\label{J2}
u^{\vee }=(u^{b}\circ \pi )\dot{\partial}_{a}.
\end{equation}
\end{definition}
If $\{s_{a}\}$ be a basis of sections of $\Gamma ( E,\pi ,M) $, then
using the above equation we have
\begin{equation*}
s_{a}^{\vee }=\dot{\partial}_{a}.
\end{equation*}
Using the locally expression of $u^{\vee }$ we can deduce
\begin{lemma}
If $u$ and $v$ are sections of $E$ and $f\in \mathcal{F}(M)$, then
\begin{equation*}
(u+v)^{\vee }=u^{\vee }+v^{\vee },\ \ \ (fu)^{\vee }=f^{\vee
}u^{\vee },\ \ \ u^{\vee }( f^{\vee }) =0.
\end{equation*}
\end{lemma}

For any $z\in \Gamma  (F,\nu ,N) $, the $\mathcal{F}( N)
$-multilinear application
\begin{equation*}
\begin{array}{rcl}
\Lambda ( F,\nu ,N) & ^{\underrightarrow{~\ \ L_{z}~\ \ }} &
\Lambda ( F,\nu ,N)%
\end{array}%
,
\end{equation*}%
defined by%
\begin{equation*}
\begin{array}{c}
L_{z}( f) =[ \Gamma ( Th\circ \rho ,h\circ \eta ) z] ( f) ,~\forall
f\in \mathcal{F}( N),
\end{array}%
\end{equation*}%
and
\begin{align*}
L_{z}\theta ( z_{1},...,z_{q}) & =\Gamma ( Th\circ \rho
,h\circ \eta ) z( \omega ( z_{1},...,z_{q}) )\nonumber \\
& -\sum_{i=1}^{q}\theta ( ( z_{1},...,[ z,z_{i}] _{F,h},...,z_{q}) )
,
\end{align*}%
for any $\theta \in \Lambda ^{q}\mathbf{\ }( F,\nu ,N) $ and $%
z_{1},...,z_{q}\in \Gamma ( F,\nu ,N) ,$ will be called \emph{the
covariant Lie derivative with respect to the section }$z.$ Also for
any $u\in \Gamma ( E,\pi ,M) $, the $\mathcal{F}( M) $-multilinear
application
\begin{equation*}
\begin{array}{rcl}
\Lambda ( E,\pi ,M) & ^{\underrightarrow{~\ \ ( g,h)
\mathcal{L}_{u}~\ \ }} & \Lambda ( E,\pi ,M)%
\end{array}%
,
\end{equation*}%
defined by%
\begin{equation*}
\begin{array}{c}
( g,h) \mathcal{L}_{u}( f) =( g,h) ^{\ast }\{ \Gamma ( Th\circ \rho
,h\circ \eta ) (\Gamma ( g,h) u) ( g^{-1},h^{-1}) ^{\ast }f\}
,~\forall f\in \mathcal{F}( M),
\end{array}%
\end{equation*}%
and
\begin{equation*}
\begin{array}{l}
( g,h) \mathcal{L}_{u}\omega ( u_{1},...,u_{q}) \\
=( g,h) ^{\ast }\{ \mathcal{L}_{\Gamma ( g,h) u}( g^{-1},h^{-1})
^{\ast }\omega( \Gamma ( g,h)
u_{1},...,\Gamma ( g,h) u_{q})\} \\
=( g,h) ^{\ast }\{ \Gamma ( Th\circ \rho ,h\circ \eta ) [ \Gamma (
g,h) u] ( g^{-1},h^{-1}) ^{\ast }\omega ( \Gamma ( g,h)
u_{1},...,\Gamma (
g,h) u_{q}) \} \\
-( g,h) ^{\ast }\{ ( g^{-1},h^{-1}) ^{\ast }\omega ( \Gamma ( g,h)
u_{1},...,[ \Gamma ( g,h) u,\Gamma ( g,h) u_{i}] _{F,h},...,\Gamma (
g,h) u_{q}) \} ,%
\end{array}%
\end{equation*}%
for any $\omega \in \Lambda ^{q}\mathbf{\ }( E,\pi ,M) $ and $%
u_{1},...,u_{q}\in \Gamma( E,\pi ,M) ,$ will be called \emph{the
covariant Lie }$( g,h) $-\emph{derivative with respect to the
section }$u.$

\begin{definition}
For any $a=1,\cdots ,r,$ we consider the real function $U^{a}$ on $E$ such
that
\begin{equation*}
U^{a}|_{\pi ^{-1}(V)}( u_{x}) =y^{a},
\end{equation*}%
where the real numbers $y^{1},\cdots ,y^{r}$ are the fibre
components of the point $u_{x}$ in the arbitrary vector local
$(m+r)$-chart $({V},s_{{V}}).$
\end{definition}
Using the above definition, we can deduce $\dot{\partial}_{b}(
U^{a})=\delta _{b}^{a}$ and $\partial _{i}( U^{a})=0$, where $a,
b\in \{1,\cdots,r\}$ and $i\in \{1,\cdots,m\}$.

Now, let $\omega =\omega _{a}s^{a}\in \Lambda ^{1}\mathbf{\ }( E,\pi
,M)$. Then we consider the real function $\hat{\omega}$ defined by
\begin{equation}
\hat{\omega}|_{\pi ^{-1}(V)}=U^{a}( \omega _{a}\circ \pi ) |_{\pi
^{-1}(V)},
\end{equation}%
where $({V},s_{{V}})$ is an arbitrary vector local $(m+r)$-chart$.$

\begin{theorem}\label{T1}
Let $u$ be a section of $( E,\pi ,M) $. Then there exists a unique
vector field $u^{c}\in \Gamma ( TE,\tau _{E},E) $, the complete $(g, h)$-lift
of $u$, satisfying the following conditions:

\textbf{i}) $u^{c}$ is $(h\circ \pi )$-related with $ \Gamma
(Th\circ \rho ,h\circ \eta ) ( \Gamma ( g,h) u) $, i.e.,
\begin{equation*}
T(h\circ \pi )(u_{v_{x}}^{c})=\{ \Gamma (Th\circ \rho ,h\circ \eta )
( \Gamma ( g,h) u) \} (h\circ \pi ( v_{x}) ) ,
\end{equation*}

\textbf{ii}) $u^{c}(\hat{\omega})=\widehat{( g,h) \mathcal{L}%
_{u}\omega },\ \ \ \ \ \  \forall \omega \in \Lambda ^{1}\mathbf{\
}( E,\pi ,M)$.
\end{theorem}

\begin{proof}
At first we let that there exists $u^{c}$ such that satisfies in (i) and
(ii). Since $u^{c}$ is a vector field on $E$, then we can write it as
follows:
\begin{equation*}
u^{c}=A^{i}\partial _{i}+B^{a}\dot{\partial}_{a},
\end{equation*}%
where $A^{i},B^{a}\in \mathcal{F}(E)$. We have
\begin{equation*}
T(h\circ \pi )({\partial _{i}}_{v_{x}})(f)=T\pi ({\partial _{i}}%
_{v_{x}})(f\circ h)={\partial _{i}}_{v_{x}}(f\circ h\circ \pi ),
\end{equation*}%
and
\begin{equation*}
T(h\circ \pi )({\dot{\partial _{a}}}_{v_{x}})(f)=T\pi ({\dot{\partial _{a}}}%
_{v_{x}})(f\circ h)={\dot{\partial _{a}}}_{v_{x}}(f\circ h\circ \pi )=0.
\end{equation*}
From two above equations we obtain
\begin{equation*}
T(h\circ \pi )({u^{c}}_{v_{x}})(f)=A^{i}(v_{x}){\partial _{i}}%
_{v_{x}}(f\circ h\circ \pi ).
\end{equation*}
On the other hand we have
\begin{align*}
& \Gamma (Th\circ \rho ,h\circ \eta )( \Gamma( g,h) u)
h\circ \pi (v_{x})(f)  \notag  \label{10} \\
& =( ( g_{c}^{\alpha }\circ \pi )( u^{c}\circ \pi ) ( \rho _{\alpha
}^{i}\circ h\circ \pi ) )( v_{x}) {\partial _{i}}_{v_{x}}(f\circ
h\circ \pi ).
\end{align*}%
Condition (i) give us
\begin{equation*}
A^{i}=( g_{c}^{\alpha }u^{c}\rho _{\alpha }^{i}\circ h) \circ \pi .
\end{equation*}
Therefore we have
\begin{equation*}
u^{c}=( ( g_{c}^{\alpha }u^{c}\rho _{\alpha }^{i}\circ h) \circ \pi)
\partial _{i}+B^{a}\dot{\partial}_{a}.
\end{equation*}
Now, let $\omega =\omega _{b}s^{b}\in \Lambda ^{1}\mathbf{\ }( E,\pi
,M) $. Then we get
\begin{equation*}
u^{c}(\hat{\omega})=U^{b}(( g_{c}^{\alpha }u^{c}\rho _{\alpha
}^{i}\circ h) \circ \pi)\partial _{i}(\omega _{b}\circ \pi
)+B^{b}(\omega _{b}\circ \pi ).
\end{equation*}
Now, let $K_{a}^{\gamma }\left( u\right) t_{\gamma }=[\Gamma \left(
g,h\right) u,\Gamma \left( g,h\right) s_{a}]_{F,h}$. Then using
(\ref{*}) we get
\begin{equation*}
\begin{array}{cl}
K_{a}^{\gamma }( u)&\!\!\!\!=( g_{c}^{\beta }u^{c}) \circ h^{-1}\rho
_{\beta }^{j}\frac{\partial g_{a}^{\gamma }}{\partial x^{i}}\circ
h^{-1}-( g_{a}^{\alpha }\circ h^{-1}) \rho _{\alpha }^{i}\frac{%
\partial ( g_{c}^{\gamma }u^{c}) }{\partial x^{i}}\circ h^{-1} \\
& +( g_{c}^{\alpha }u^{c}) \circ h^{-1}L_{\alpha \beta }^{\gamma
}( g_{a}^{\beta }\circ h^{-1}).%
\end{array}%
\end{equation*}
On the other hand we have%
\begin{equation*}
\begin{array}{cl}
( g,h) \mathcal{L}_{u}\omega (s_{a}) &\!\!\!\!=( g,h) ^{\ast }\{
\Gamma ( Th\circ \rho ,h\circ \eta ) (\Gamma ( g,h) u)(
g^{-1},h^{-1}) ^{\ast }\omega (
\Gamma ( g,h) s_{a}) \} \\
& -( g,h) ^{\ast }\{ ( g^{-1},h^{-1}) ^{\ast
}\omega ( K_{a}^{\gamma }( u) t_{\gamma }) \} \\
&\!\!\!\!=g_{b}^{\alpha }u^{b}( \rho _{\alpha }^{i}\circ h) \frac{%
\partial \omega _{a}}{\partial x^{i}}-\tilde{g}_{\gamma }^{b}\omega
_{b}K_{a}^{\gamma }( u) \circ h \\
&\!\!\!\!=g_{b}^{\alpha }u^{b}( \rho _{\alpha }^{i}\circ h) \frac{%
\partial \omega _{a}}{\partial x^{i}}-g_{c}^{\beta }u^{c}( \rho _{\beta
}^{j}\circ h) \frac{\partial g_{a}^{\gamma }}{\partial x^{i}}\tilde{g}%
_{\gamma }^{b}\omega _{b} \\
& -g_{a}^{\alpha }( \rho _{\alpha }^{i}\circ h) \frac{\partial (
g_{c}^{\gamma }u^{c}) }{\partial x^{i}}\tilde{g}_{\gamma }^{b}\omega
_{b}+g_{c}^{\alpha }u^{c}( L_{\alpha \beta }^{\gamma }\circ
h) g_{a}^{\beta }\tilde{g}_{\gamma }^{b}\omega _{b}.%
\end{array}%
\end{equation*}
Thus we have%
\begin{equation*}
\begin{array}{cl}
\widehat{( g,h) \mathcal{L}_{u}\omega } &\!\!\!\!=U^{a}( (
g,h) \mathcal{L}_{u}\omega (s_{a})) \circ \pi \\
&\!\!\!\!=U^{a}( g_{b}^{\alpha }u^{b}( \rho _{\alpha }^{i}\circ h)
\frac{\partial \omega _{a}}{\partial x^{i}}-\tilde{g}_{\gamma
}^{b}\omega
_{b}K_{a}^{\gamma }( u) \circ h) \circ \pi \\
&\!\!\!\!=U^{a}\{ ( g_{b}^{\alpha }u^{b}( \rho _{\alpha }^{i}\circ
h) \frac{\partial \omega _{a}}{\partial x^{i}}-\tilde{g}_{\gamma
}^{b}\omega _{b}g_{c}^{\beta }u^{c}( \rho _{\beta }^{j}\circ h)
\frac{\partial g_{a}^{\gamma }}{\partial x^{i}}) \circ \pi  \\
&+( -g_{a}^{\alpha }( \rho _{\alpha }^{i}\circ h)
\frac{\partial ( g_{c}^{\gamma }u^{c}) }{\partial x^{i}}\tilde{g}%
_{\gamma }^{b}\omega _{b}+g_{c}^{\alpha }u^{c}( L_{\alpha \beta
}^{\gamma }\circ h) g_{a}^{\beta }\tilde{g}_{\gamma }^{b}\omega _{b}%
) \circ \pi \} .%
\end{array}%
\end{equation*}%
But condition (ii) gives us
\begin{equation*}
B^{b}(\omega _{b}\circ \pi )=-U^{a}(K_{a}^{\gamma }( u) \circ h\circ
\pi )(( \tilde{g}_{\gamma }^{b}\omega _{b}) \circ \pi).
\end{equation*}
Since $\omega $ is arbitrary, then we suppose that $\omega =s^{b}$.
Thus we have $\omega _{b}=1$ and $\omega _{a}=0$, for any $a\neq b$.
Therefore we obtain
\begin{equation*}
B^{b}=-U^{a}( K_{a}^{\gamma }( u) \circ h\circ \pi ) (
\tilde{g}_{\gamma }^{b}\circ \pi ) .
\end{equation*}
So, for $u^{c}$ we can obtain the following locally expression:%
\begin{equation}\label{J3}
\begin{array}{cl}
u^{c}&\!\!\!\!=( g_{e}^{\alpha }u^{e}\rho _{\alpha }^{i}\circ h)
\circ \pi
\partial _{i}-U^{a}( K_{a}^{\gamma }( u) \circ
h\circ
\pi) ( \tilde{g}_{\gamma }^{b}\circ \pi ) \dot{\partial}%
_{b}.%
\end{array}%
\end{equation}
The above relation prove the existence and uniqueness of the complete lift.
\end{proof}
\begin{definition}
The complete $( g,h) $-lift of a function $f\in\mathcal{F}(N)$ into $\mathcal{F}(E)$ is the function
\begin{equation*}
f^{c}:E\longrightarrow \mathbb{R},
\end{equation*}%
defined by
\begin{equation*}\label{J4}
f^{c}|_{\pi ^{-1}(V)}=U^{a}( g_{a}^{b}\circ \pi ) (\rho
_{b}^{i}\circ h\circ \pi )\partial _{i}(f\circ h\circ \pi )|_{\pi
^{-1}(V)},
\end{equation*}%
where $({V},s_{{V}})$ is an arbitrary vector local $(m+r)$-chart.
\end{definition}
\begin{lemma}\label{12}
If $u$ is a section of $( E,\pi ,M) $ and $f, f_1, f_2\in \mathcal{F}(N)$%
, then%
\begin{equation*}
\begin{array}{l}
(i)\ (f_1+f_2)^{c}=f_1^{c}+f_2^{c}, \\
(ii)\ (f_1f_2)^{c}=f_1^{c}f_2^{\vee }+f_1^{\vee }f_2^{c}, \\
(iii)\ u^{\vee }( f^{c}) =\{ \Gamma (Th\circ \rho ,h\circ
\eta )(\Gamma( g,h) u)(f)\} ^{\vee }.%
\end{array}%
\end{equation*}
\end{lemma}
\begin{proof}
We only prove (iii). Using (\ref{J1}), (\ref{J2}) and (\ref{J4}) we
obtain
\begin{align*}
u^{\vee }( f^{c}) & =(u^{b}\circ \pi )\dot{\partial}_{b}(U^{b}(
g_{b}^{\alpha }\circ \pi ) (\rho _{\alpha }^{i}\circ
h\circ \pi )\partial _{i}(f\circ h\circ \pi ))\\
& =( g_{b}^{\alpha }\circ \pi ) (u^{b}\circ \pi )(\rho _{\alpha
}^{i}\circ h\circ \pi )\partial _{i}(f\circ h\circ \pi ) \\
& =\{ ( g_{b}^{\alpha }\circ h^{-1}) ( u^{b}\circ h^{-1}) \rho
_{\alpha }^{i}(( \partial _{i}(f\circ
h)) \circ h^{-1})\} \circ h\circ \pi \\
& =\{ \Gamma (Th\circ \rho ,h\circ \eta )(\Gamma ( g,h) u) (f)\}
^{\vee }.
\end{align*}
\end{proof}
\begin{lemma}\label{13}
Let $u$ be a section of $( E,\pi ,M) $. Then
\begin{equation*}
u^{c}( f^{c}) =\{ (\Gamma (Th\circ \rho ,h\circ \eta )(\Gamma ( g,h)
u) (f)\} ^{c},\ \ \ \forall f\in \mathcal{F}(N).
\end{equation*}
\end{lemma}

\begin{proof}
Using (\ref{J3}) and (\ref{J4}) we get%
\begin{equation*}
\begin{array}{l}
\{ (\Gamma (Th\circ \rho ,h\circ \eta )( \Gamma ( g,h) u%
) (f)\} ^{c} \\
=( ( g_{e}^{\alpha }\circ h^{-1}) ( u^{e}\circ h^{-1}) \rho _{\alpha
}^{i}\partial _{i}(f\circ h)\circ h^{-1})
^{c} \\
=U^{b}( g_{b}^{\beta }\circ \pi ) (\rho _{\beta }^{j}\circ h\circ
\pi )\partial _{j}(( g_{e}^{\alpha }\circ \pi )( u^{e}\circ \pi ) (
\rho _{\alpha }^{i}\circ h\circ \pi)
\partial _{i}(f\circ h\circ \pi ))\\
=U^{b}( g_{b}^{\beta }\circ \pi ) (\rho _{\beta }^{j}\circ h\circ
\pi )(\partial _{j}( g_{e}^{\alpha }\circ \pi ))( u^{e}\circ \pi ) (
\rho _{\alpha }^{i}\circ h\circ \pi
) \partial _{i}(f\circ h\circ \pi ) \\
+U^{b}( g_{b}^{\beta }\circ \pi ) (\rho _{\beta }^{j}\circ h\circ
\pi )( g_{e}^{\alpha }\circ \pi ) ( \partial _{j}( u^{e}\circ \pi )
)( \rho _{\alpha }^{i}\circ h\circ \pi
) \partial _{i}(f\circ h\circ \pi ) \\
+U^{b}( g_{b}^{\beta }\circ \pi ) (\rho _{\beta }^{j}\circ h\circ
\pi )( g_{e}^{\alpha }\circ \pi ) ( u^{e}\circ \pi )
(\partial _{j}( \rho _{\alpha }^{i}\circ h\circ \pi ) %
)\partial _{i}(f\circ h\circ \pi ) \\
+U^{b}( g_{b}^{\beta }\circ \pi ) (\rho _{\beta }^{j}\circ h\circ
\pi )( g_{e}^{\alpha }\circ \pi) ( u^{e}\circ \pi ) ( \rho _{\alpha
}^{i}\circ h\circ \pi ) \partial _{j}\partial
_{i}(f\circ h\circ \pi ).%
\end{array}%
\end{equation*}
Again, (\ref{J3}) and (\ref{J4}) give us%
\begin{equation*}
\begin{array}{l}
u^{c}( f^{c}) =( g_{e}^{\alpha }\circ \pi ) (u^{e}\circ \pi )( \rho
_{\alpha }^{i}\circ h\circ \pi )
\partial _{i}( U^{b}( g_{b}^{\beta }\circ \pi)
(\rho _{\beta
}^{j}\circ h\circ \pi )\partial _{j}(f\circ h\circ \pi )) \\
-U^{a}( K_{a}^{\gamma }( u) \circ h\circ \pi) ( \tilde{g}_{\gamma
}^{b}\circ \pi ) \dot{\partial}_{b}( U^{b}( g_{b}^{\beta }\circ \pi)
(\rho _{\beta }^{j}\circ h\circ
\pi )\partial _{j}(f\circ h\circ \pi )) \\
=U^{b}( g_{e}^{\alpha }\circ \pi ) (u^{e}\circ \pi )( \rho _{\alpha
}^{i}\circ h\circ \pi )( \partial _{i}( g_{b}^{\beta }\circ \pi ) )
(\rho _{\beta }^{j}\circ h\circ \pi
)\partial _{j}(f\circ h\circ \pi ) \\
+U^{b}( g_{e}^{\alpha }\circ \pi) (u^{e}\circ \pi )( \rho _{\alpha
}^{i}\circ h\circ \pi ) ( g_{b}^{\beta }\circ \pi ) (
\partial _{i}(\rho _{\beta }^{j}\circ h\circ \pi ))
\partial _{j}(f\circ h\circ \pi ) \\
+U^{b}( g_{e}^{\alpha }\circ \pi ) (u^{e}\circ \pi )( \rho _{\alpha
}^{i}\circ h\circ \pi ) ( g_{b}^{\beta }\circ \pi ) (\rho _{\beta
}^{j}\circ h\circ \pi )\partial _{i}\partial
_{j}(f\circ h\circ \pi ) \\
-U^{a}( K_{a}^{\beta }( u) \circ h\circ \pi ) (\rho
_{\beta }^{j}\circ h\circ \pi )\partial _{j}(f\circ h\circ \pi ) \\
=U^{b}( g_{e}^{\alpha }\circ \pi ) (u^{e}\circ \pi )( \rho _{\alpha
}^{i}\circ h\circ \pi ) ( \partial _{i}( g_{b}^{\beta }\circ \pi ) )
(\rho _{\beta }^{j}\circ h\circ \pi
)\partial _{j}(f\circ h\circ \pi ) \\
+U^{b}( g_{e}^{\alpha }\circ \pi ) (u^{e}\circ \pi )( \rho _{\alpha
}^{i}\circ h\circ \pi ) ( g_{b}^{\beta }\circ \pi ) (
\partial _{i}(\rho _{\beta }^{j}\circ h\circ \pi ))
\partial _{j}(f\circ h\circ \pi ) \\
+U^{b}( g_{e}^{\alpha }\circ \pi ) (u^{e}\circ \pi )( \rho _{\alpha
}^{i}\circ h\circ \pi ) ( g_{b}^{\beta }\circ \pi ) (\rho _{\beta
}^{j}\circ h\circ \pi )\partial _{i}\partial
_{j}(f\circ h\circ \pi ) \\
-U^{b}( ( g_{c}^{\beta }u^{c}) \circ \pi ( \rho _{\beta }^{j}\circ
h\circ \pi ) \partial _{j}( g_{b}^{\gamma }\circ \pi ) ) (\rho
_{\gamma }^{k}\circ h\circ \pi )\partial _{k}(f\circ
h\circ \pi ) \\
+U^{b}( g_{b}^{\alpha }\circ \pi ) \rho _{\alpha }^{i}\circ
h\circ \pi ( \partial _{i}( g_{c}^{\gamma }u^{c}) \circ \pi %
) (\rho _{\gamma }^{k}\circ h\circ \pi )\partial _{k}(f\circ h\circ
\pi ) \\
-U^{b}( g_{c}^{\alpha }u^{c}) \circ \pi ( L_{\alpha \beta }^{\gamma
}\circ h\circ \pi ) ( g_{b}^{\beta }\circ \pi )
(\rho _{\gamma }^{k}\circ h\circ \pi )\partial _{k}(f\circ h\circ \pi ).%
\end{array}%
\end{equation*}
Using (\ref{P1}) in the above equation, the proof completes.
\end{proof}

\begin{definition}
The \emph{complete }$( g,h) $-lift $u^{C}$ of a section $%
u\in \Gamma (E,\pi ,M)$ is the section of $( (\rho ,\eta )TE,(\rho
,\eta )\tau _{E},E) $ given by%
\begin{equation}  \label{com}
\begin{array}{cl}
u^{C} & =( g_{e}^{\alpha }\circ \pi ) (u^{e}\circ \pi )T_{\alpha
}\oplus ( g_{e}^{\alpha }\circ \pi ) ( u^{e}\circ \pi
) ( \rho _{\alpha }^{i}\circ h\circ \pi ) \partial _{i} \\
& -U^{a}( K_{a}^{\gamma }( u) \circ h\circ \pi ) ( \tilde{g}_{\gamma
}^{b}\circ \pi ) ( 0_{( h\circ \pi
) ^{\ast }E}\oplus \dot{\partial}_{b}) \\
& =( ( g_{e}^{\alpha }\circ \pi ) (u^{e}\circ \pi )) T_{\alpha
}\oplus ( \rho _{\alpha }^{i}\circ h\circ \pi )
\partial _{i} \\
& -U^{a}(K_{a}^{\gamma }( u) \circ h\circ \pi ) ( \tilde{g}_{\gamma
}^{b}\circ \pi ) ( 0_{( h\circ \pi
) ^{\ast }E}\oplus \dot{\partial}_{b}) \\
& =( ( g_{e}^{\alpha }\circ \pi ) (u^{e}\circ \pi ))
\tilde{\partial}_{\alpha }-U^{a}( K_{a}^{\gamma }( u) \circ
h\circ \pi ) ( \tilde{g}_{\gamma }^{b}\circ \pi ) \dot{%
\tilde{\partial}}_{b}.%
\end{array}%
\end{equation}
\end{definition}
Using the above definition, we can obtain
\begin{equation*}
\Gamma ( \tilde{\rho},Id_{E}) (u^{C})=u^{c}.
\end{equation*}

In the particular case of Lie algebroids, $( g,\eta ,h) =(
Id_{E},Id_{M},Id_{M}) ,$ the complete lifts are given by the
equality:
\begin{equation*}
u^{c}=\{(u^{a}\rho _{a}^{i})\circ \pi \}\partial _{i}+y^{b}\{(\rho
_{b}^{i}\partial _{i}u^{a}+u^{d}L_{bd}^{a})\circ \pi \}\dot{\partial}%
_{a},
\end{equation*}%
\begin{equation*}
u^{C}=(u^{a}\circ \pi )\tilde{\partial}_{a}+y^{b}\{(\rho _{b}^{i}\partial
_{i}u^{a}+u^{d}L_{bd}^{a})\circ \pi \}\dot{\tilde{\partial}}_{a},
\end{equation*}%
and in the classical case, $\rho =Id_{TM},$ the complete lifts are given by
the equality:%
\begin{equation*}
u^{C}=(X^{i}\circ \pi )\partial _{i}+y^{j}(\partial _{j}X^{i}\circ \pi )\dot{%
\partial}_{i}=u^{c}.
\end{equation*}
\begin{definition}
If $u=u^{a}s_{a}$ is a section of $( E,\pi ,M) $, then we
introduce the vertical $( g,h) $-lift of $u$ as section of $%
( ( \rho ,\eta ) TE,( \rho ,\eta ) \tau _{E},E) $ given by
\begin{equation*}
u^{V}=0_{(h\circ \pi )^{\ast }E}\oplus u^{\vee }.
\end{equation*}
\end{definition}
If $u=u^{a}e_{a}\in \Gamma (E,\pi ,M)$, then in the locally
expressions we get
\begin{align}\label{L4}
u^{V}&=0_{(h\circ \pi )^{\ast }E}\oplus (u^{a}\circ \pi
)\dot{\partial}_{a}=(u^{a}\circ \pi )(0_{(h\circ \pi )^{\ast
}E}\oplus \dot{\partial}_{a})=(u^{a}\circ \pi
)\dot{\tilde{\partial}}_{a}.
\end{align}
In particular, we have $s_{a}^{V}=\dot{\tilde{\partial}}_{a}$.
\begin{remark}
Using the almost tangent $( g,h) $-structure $\mathcal{J}_{( g,h) }$ given by\textrm{\ }%
\begin{equation*}
\begin{array}{rcl}
\Gamma ( ( \rho ,\eta ) TE,( \rho ,\eta ) \tau _{E},E) &
^{\underrightarrow{\mathcal{J}_{( g,h) }}} & \Gamma ( ( \rho ,\eta )
TE,( \rho ,\eta ) \tau
_{E},E),\\
Z^{\alpha }\tilde{\partial}_{\alpha
}+Y^{b}{\dot{{\tilde{\partial}}}} _{b} & \longmapsto & (
\tilde{g}_{\alpha }^{b}\circ \pi ) Z^{\alpha
}{\dot{{\tilde{\partial}}}}_{b},
\end{array}%
\end{equation*}%
it results that $\mathcal{J}_{( g,h) }( u^{C}) =u^{V}.$
\end{remark}

\begin{theorem}\label{KM}
The Lie brackets of vertical and complete $( g,h) $-lifts satisfy
the following equalities:
\begin{eqnarray*}
i)\ [u^{V},v^{V}]_{(\rho ,\eta )TE} &=&0,\ \  \\
ii)\ [u^{V},v^{C}]_{(\rho ,\eta )TE} &=&\{ \Gamma ( g^{-1},h^{-1})
[\Gamma ( g,h) u,\Gamma ( g,h)
v]_{F,h}\} ^{V}, \\
iii)\ [u^{C},v^{C}]_{(\rho ,\eta )TE} &=&\{ \Gamma ( g^{-1},h^{-1})
[\Gamma ( g,h) u,\Gamma ( g,h) v]_{F,h}\} ^{C}.
\end{eqnarray*}
\end{theorem}

\begin{proof}
Direct calculation gives us
\begin{align}
\lbrack u^{V},v^{C}]_{(\rho ,\eta )TE}& =-\{((u^{a}\tilde{g}_{\gamma
}^{b})\circ \pi )(((g_{c}^{\beta }v^{c}(\rho _{\beta }^{j}\circ
h))\circ
\pi )\partial _{j}(g_{a}^{\gamma }\circ \pi )  \notag  \label{Ar1} \\
& \ \ \ -((g_{a}^{\alpha }(\rho _{\alpha }^{i}\circ h))\circ \pi )\partial
_{i}((g_{c}^{\gamma }v^{c})\circ \pi )+((g_{c}^{\alpha }v^{c}g_{a}^{\beta
}(L_{\alpha \beta }^{\gamma }\circ h))\circ \pi ))  \notag \\
& \ \ \ +((g_{c}^{\alpha }v^{c}(\rho _{\alpha }^{i}\circ h))\circ
\pi )\partial _{i}(u^{b}\circ \pi )\}\dot{\tilde{\partial}}_{b}.
\end{align}%
On the other hand, we have
\begin{align*}
& [\Gamma (g,h)(u),\Gamma (g,h)(v)]_{F,h}=[((g_{a}^{\alpha
}u^{a})\circ
h^{-1})t_{\alpha },((g_{b}^{\beta }v^{b})\circ h^{-1})t_{\beta }]_{F,h} \\
& =((g_{a}^{\alpha }u^{a})\circ h^{-1})\Gamma (Th\circ \rho ,h\circ \eta
)(t_{\alpha })((g_{b}^{\beta }v^{b})\circ h^{-1})t_{\beta } \\
& \ \ \ -((g_{b}^{\beta }v^{b})\circ h^{-1})\Gamma (Th\circ \rho ,h\circ
\eta )(t_{\beta })((g_{a}^{\alpha }u^{a})\circ h^{-1})t_{\alpha } \\
& \ \ \ +((g_{a}^{\alpha }u^{a}g_{b}^{\beta }v^{b})\circ h^{-1})L_{\alpha
\beta }^{\gamma }t_{\gamma } \\
& =\{((g_{a}^{\alpha }u^{a})\circ h^{-1})\rho _{\alpha }^{i}\frac{%
\partial (g_{b}^{\gamma }v^{b})}{\partial x^{i}}\circ h^{-1}-((g_{b}^{\beta
}v^{b})\circ h^{-1})\rho _{\beta }^{i}\frac{\partial (g_{a}^{\gamma }u^{a})}{%
\partial x^{i}}\circ h^{-1} \\
& \ \ \ +((g_{a}^{\alpha }u^{a}g_{b}^{\beta }v^{b})\circ
h^{-1})L_{\alpha \beta }^{\gamma }\}t_{\gamma }.
\end{align*}%
The above equation gives us
\begin{align*}
& \Gamma (g^{-1},h^{-1})([\Gamma (g,h)(u),\Gamma (g,h)(v)]_{F,h}=\tilde{g}%
_{\gamma }^{d}\{g_{a}^{\alpha }u^{a}(\rho _{\alpha }^{i}\circ h)\frac{%
\partial (g_{b}^{\gamma }v^{b})}{\partial x^{i}} \\
& \ \ \ -g_{b}^{\beta }v^{b}(\rho _{\beta }^{i}\circ
h)\frac{\partial (g_{a}^{\gamma }u^{a})}{\partial
x^{i}}+g_{a}^{\alpha }u^{a}g_{b}^{\beta }v^{b}(L_{\alpha \beta
}^{\gamma }\circ h)\}s_{d}.
\end{align*}%
Thus
\begin{align}
& (\Gamma (g^{-1},h^{-1})[\Gamma (g,h)(u),\Gamma (g,h)(v)]_{F,h})%
^{V}=(\tilde{g}_{\gamma }^{d}\circ \pi )\{((v^{b}g_{a}^{\alpha
}u^{a}(\rho _{\alpha }^{i}\circ h))\circ \pi )\partial
_{i}(g_{b}^{\gamma
}\circ \pi )  \notag  \label{Ar2} \\
& \ \ \ +((g_{b}^{\gamma }g_{a}^{\alpha }u^{a}(\rho _{\alpha }^{i}\circ
h))\circ \pi )\partial _{i}(v^{b}\circ \pi )-((u^{a}g_{b}^{\beta }v^{b}(\rho
_{\beta }^{i}\circ h))\circ \pi )\partial _{i}(g_{a}^{\gamma }\circ \pi )
\notag \\
& \ \ \ -((g_{a}^{\gamma }g_{b}^{\beta }v^{b}(\rho _{\beta }^{i}\circ
h))\circ \pi )\partial _{i}(u^{a}\circ \pi )+(g_{a}^{\alpha
}u^{a}g_{b}^{\beta }v^{b}(L_{\alpha \beta }^{\gamma }\circ h))\circ \pi %
\}\dot{\tilde{\partial}}_{d}.
\end{align}%
From (\ref{Ar1}) and (\ref{Ar2}) we get (ii). Now we prove (iii). we
have
\begin{equation}
\{\Gamma (g^{-1},h^{-1})[\Gamma (g,h)u,\Gamma (g,h)v]_{F,h}\}^{C}=\mathcal{A}%
^{\gamma }{\tilde{\partial}}_{\gamma }+\mathcal{B}^{r}{\dot{\tilde{\partial}}%
_{r}},  \label{lili00}
\end{equation}%
where
\begin{align}
\mathcal{A}^{\gamma }& =((g_{e}^{\alpha }u^{e}(\rho _{\alpha }^{i}\circ
h))\circ \pi )\partial _{i}((g_{c}^{\lambda }v^{c})\circ \pi
)-((g_{c}^{\beta }v^{c}(\rho _{\beta }^{i}\circ h))\circ \pi )\partial
_{i}((g_{e}^{\gamma }u^{e})\circ \pi )  \notag \\
& \ \ \ +((g_{e}^{\alpha }u^{e}g_{c}^{\beta }v^{c}(L_{\alpha \beta }^{\gamma
}\circ h)\circ \pi ),
\end{align}%
and
\begin{align}
\mathcal{B}^{r}& =-U^{d}(\tilde{g}_{\gamma }^{r}\circ \pi )\Big\{%
((g_{a}^{\alpha }u^{a}((\rho _{\alpha }^{i}\rho _{\lambda }^{j})\circ
h))\circ \pi )\partial _{i}((g_{b}^{\lambda }v^{b})\circ \pi )\partial
_{j}(g_{d}^{\gamma }\circ \pi )  \notag \\
& \ \ \ -((g_{b}^{\beta }v^{b}((\rho _{\beta }^{i}\rho _{\lambda }^{j})\circ
h))\circ \pi )\partial _{i}((g_{a}^{\lambda }u^{a})\circ \pi )\partial
_{j}(g_{d}^{\gamma }\circ \pi )  \notag \\
& \ \ \ \ +((g_{a}^{\alpha }u^{a}g_{b}^{\beta }v^{b}((L_{\alpha \beta
}^{\lambda }\rho _{\lambda }^{j})\circ h))\circ \pi )\partial
_{j}(g_{d}^{\gamma }\circ \pi )  \notag \\
& \ \ \ -((g_{d}^{\lambda }((\rho _{\lambda }^{j}\rho _{\alpha }^{i})\circ
h))\circ \pi )\partial _{j}((g_{a}^{\alpha }u^{a})\circ \pi )\partial
_{i}((g_{b}^{\gamma }v^{b})\circ \pi )  \notag \\
& \ \ \ -((g_{d}^{\lambda }g_{a}^{\alpha }u^{a}(\rho _{\lambda }^{j}\circ
h))\circ \pi )\partial _{j}(\rho _{\alpha }^{i}\circ h\circ \pi )\partial
_{i}((g_{b}^{\gamma }v^{b})\circ \pi )  \notag \\
& \ \ \ \ -((g_{d}^{\lambda }g_{a}^{\alpha }u^{a}((\rho _{\alpha }^{i}\rho
_{\lambda }^{j})\circ h))\circ \pi )\partial _{j}\partial
_{i}((g_{b}^{\gamma }v^{b})\circ \pi )  \notag \\
& \ \ \ \ +((g_{d}^{\lambda }((\rho _{\lambda }^{j}\rho _{\beta }^{i})\circ
h))\circ \pi )\partial _{j}((g_{b}^{\beta }v^{b})\circ \pi )\partial
_{i}((g_{a}^{\gamma }u^{a})\circ \pi )  \notag \\
& \ \ \ \ +((g_{d}^{\lambda }g_{b}^{\beta }v^{b}(\rho _{\lambda }^{j}\circ
h))\circ \pi )\partial _{j}(\rho _{\beta }^{i}\circ h\circ \pi )\partial
_{i}((g_{a}^{\gamma }u^{a})\circ \pi )  \notag \\
& \ \ \ \ +((g_{d}^{\lambda }g_{b}^{\beta }v^{b}((\rho _{\beta }^{i}\rho
_{\lambda }^{j})\circ h))\circ \pi )\partial _{j}\partial
_{i}((g_{a}^{\gamma }u^{a})\circ \pi )  \notag \\
& \ \ \ \ -((g_{d}^{\lambda }g_{b}^{\beta }v^{b}((L_{\alpha \beta }^{\gamma
}\rho _{\lambda }^{j})\circ h))\circ \pi )\partial _{j}((g_{a}^{\alpha
}u^{a})\circ \pi )  \notag \\
& \ \ \ \ -((g_{d}^{\lambda }g_{a}^{\alpha }u^{a}((L_{\alpha \beta }^{\gamma
}\rho _{\lambda }^{j})\circ h))\circ \pi )\partial _{j}((g_{b}^{\beta
}v^{b})\circ \pi )  \notag \\
& \ \ \ \ -((g_{d}^{\lambda }g_{a}^{\alpha }u^{a}g_{b}^{\beta }v^{b}(\rho
_{\lambda }^{j}\circ h))\circ \pi )\partial _{j}(L_{\alpha \beta }^{\gamma
}\circ h\circ \pi )  \notag \\
& \ \ \ \ +((g_{d}^{\mu }g_{a}^{\alpha }u^{a}((L_{\lambda \mu }^{\gamma
}\rho _{\alpha }^{i})\circ h))\circ \pi )\partial _{i}((g_{b}^{\lambda
}v^{b})\circ \pi )  \notag \\
& \ \ \ \ -((g_{d}^{\mu }g_{b}^{\beta }v^{b}((L_{\lambda \mu }^{\gamma }\rho
_{\beta }^{i})\circ h))\circ \pi )\partial _{i}((g_{a}^{\lambda }u^{a})\circ
\pi )  \notag \\
& \ \ \ \ +((g_{d}^{\mu }g_{a}^{\alpha }u^{a}g_{b}^{\beta }v^{b}((L_{\lambda
\mu }^{\gamma }L_{\alpha \beta }^{\lambda })\circ h))\circ \pi )\Big\}.
\end{align}
Using (\ref{com}) and direct calculation we get
\begin{equation}
\lbrack u^{C},v^{C}]_{(\rho ,\eta )TE}=\mathcal{A}^{\gamma }{\tilde{\partial}%
}_{\gamma }+(\mathcal{B}^{r}+\mathcal{C}^{r}){\dot{\tilde{\partial}}_{r}},
\label{lili000}
\end{equation}%
where
\begin{align}
\mathcal{C}^{r}& =-U^{d}(\tilde{g}_{\gamma }^{r}\circ \pi )\Big\{-((\tilde{g}%
_{\alpha }^{a}g_{c}^{\lambda }u^{c}g_{e}^{\sigma }v^{e}((\rho _{\sigma
}^{i}\rho _{\lambda }^{j})\circ h))\circ \pi )\partial _{j}(g_{d}^{\alpha
}\circ \pi )\partial _{i}(g_{a}^{\gamma }\circ \pi )  \notag  \label{lili0}
\\
& \ \ \ +((\tilde{g}_{\beta }^{b}g_{c}^{\lambda }v^{c}g_{e}^{\sigma
}u^{e}((\rho _{\sigma }^{i}\rho _{\lambda }^{j})\circ h))\circ \pi )\partial
_{j}(g_{d}^{\beta }\circ \pi )\partial _{i}(g_{b}^{\gamma }\circ \pi )
\notag \\
& \ \ \ +((\tilde{g}_{\alpha }^{a}g_{d}^{\lambda }g_{e}^{\sigma }v^{e}((\rho
_{\sigma }^{i}\rho _{\lambda }^{j})\circ h))\circ \pi )\partial
_{j}((g_{c}^{\alpha }u^{c})\circ \pi )\partial _{i}(g_{a}^{\gamma }\circ \pi
)  \notag \\
& \ \ \ -((\tilde{g}_{\beta }^{b}g_{d}^{\lambda }g_{e}^{\sigma }u^{e}((\rho
_{\sigma }^{i}\rho _{\lambda }^{j})\circ h))\circ \pi )\partial
_{j}((g_{c}^{\beta }v^{c})\circ \pi )\partial _{i}(g_{b}^{\gamma }\circ \pi )
\notag \\
& \ \ \ -((\tilde{g}_{\alpha }^{a}g_{c}^{\lambda }u^{c}g_{e}^{\sigma
}v^{e}((L_{\lambda \mu }^{\alpha }\rho _{\sigma }^{i})\circ h))\circ \pi
)\partial _{i}(g_{a}^{\gamma }\circ \pi )  \notag \\
& \ \ \ +((\tilde{g}_{\beta }^{b}g_{c}^{\lambda }v^{c}g_{e}^{\sigma
}u^{e}((L_{\lambda \mu }^{\beta }\rho _{\sigma }^{i})\circ h))\circ \pi
)\partial _{i}(g_{b}^{\gamma }\circ \pi )\Big\}  \notag \\
& \ \ \ -U^{d}((g_{e}^{\alpha }u^{e}g_{c}^{\lambda }v^{c}((\rho _{\alpha
}^{i}\rho _{\lambda }^{j})\circ h))\circ \pi )\partial _{i}(\tilde{g}%
_{\gamma }^{r}\circ \pi )\partial _{j}(g_{d}^{\gamma }\circ \pi )  \notag \\
& \ \ \ +U^{d}((g_{e}^{\alpha }u^{e}g_{d}^{\lambda }((\rho _{\alpha
}^{i}\rho _{\lambda }^{j})\circ h))\circ \pi )\partial _{i}(\tilde{g}%
_{\gamma }^{r}\circ \pi )\partial _{j}((g_{c}^{\gamma }v^{c})\circ \pi )
\notag \\
& \ \ \ -U^{d}((g_{e}^{\alpha }u^{e}g_{c}^{\lambda }v^{c}g_{d}^{\mu
}\partial _{i}(\tilde{g}_{\gamma }^{r}\circ \pi )((L_{\lambda \mu }^{\gamma
}\rho _{\alpha }^{i})\circ h))\circ \pi )  \notag \\
& \ \ \ +U^{d}((g_{e}^{\beta }v^{e}g_{c}^{\lambda }u^{c}((\rho _{\beta
}^{i}\rho _{\lambda }^{j})\circ h))\circ \pi )\partial _{i}(\tilde{g}%
_{\gamma }^{r}\circ \pi )\partial _{j}(g_{d}^{\gamma }\circ \pi )  \notag \\
& \ \ \ -U^{d}((g_{e}^{\beta }v^{e}g_{d}^{\lambda }((\rho _{\beta }^{i}\rho
_{\lambda }^{j})\circ h))\circ \pi )\partial _{i}(\tilde{g}_{\gamma
}^{r}\circ \pi )\partial _{j}((g_{c}^{\gamma }u^{c})\circ \pi )  \notag \\
& \ \ \ +U^{d}((g_{e}^{\beta }v^{e}g_{c}^{\lambda }u^{c}g_{d}^{\mu }\partial
_{i}(\tilde{g}_{\gamma }^{r}\circ \pi )((L_{\lambda \mu }^{\gamma }\rho
_{\beta }^{i})\circ h))\circ \pi ).
\end{align}%
On the other hand, we have
\begin{equation*}
g_{d}^{\gamma }=\tilde{g}_{\alpha }^{a}g_{a}^{\gamma }g_{d}^{\alpha
}.
\end{equation*}%
Derivative of the above expression with respect to $j$, we get
\begin{equation*}
\partial _{j}(g_{d}^{\gamma })=-\partial _{j}(\tilde{g}_{\alpha
}^{a})g_{a}^{\gamma }g_{d}^{\alpha }.
\end{equation*}%
Using the above equation, we obtain
\begin{align}
& -U^{d}((g_{e}^{\alpha }u^{e}g_{c}^{\lambda }v^{c}((\rho _{\alpha }^{i}\rho
_{\lambda }^{j})\circ h))\circ \pi )\partial _{i}(\tilde{g}_{\gamma
}^{r}\circ \pi )\partial _{j}(g_{d}^{\gamma }\circ \pi )  \notag \\
& \ \ \ =U^{d}((g_{e}^{\alpha }u^{e}g_{c}^{\lambda }v^{c}((\rho _{\alpha
}^{i}\rho _{\lambda }^{j})\circ h))\circ \pi )\partial _{i}(\tilde{g}%
_{\gamma }^{r}\circ \pi )\partial _{j}(\tilde{g}_{\alpha }^{a}\circ \pi
)(g_{a}^{\gamma }\circ \pi )(g_{d}^{\alpha }\circ \pi )  \notag \\
& \ \ \ =-U^{d}((g_{e}^{\alpha }u^{e}g_{c}^{\lambda }v^{c}((\rho _{\alpha
}^{i}\rho _{\lambda }^{j})\circ h))\circ \pi )\partial _{i}(\tilde{g}%
_{\gamma }^{r}\circ \pi )\tilde{g}_{\alpha }^{a}\partial _{j}(g_{a}^{\gamma
}\circ \pi )(g_{d}^{\alpha }\circ \pi )  \notag \\
& \ \ \ =U^{d}((g_{e}^{\alpha }u^{e}g_{c}^{\lambda }v^{c}((\rho _{\alpha
}^{i}\rho _{\lambda }^{j})\circ h))\circ \pi )(\tilde{g}_{\gamma }^{r}\circ
\pi )\partial _{i}(\tilde{g}_{\alpha }^{a}\circ \pi )\partial
_{j}(g_{a}^{\gamma }\circ \pi )(g_{d}^{\alpha }\circ \pi )  \notag \\
& \ \ \ =U^{d}((g_{e}^{\alpha }u^{e}g_{c}^{\lambda }v^{c}((\rho _{\alpha
}^{i}\rho _{\lambda }^{j})\circ h))\circ \pi )(\tilde{g}_{\gamma }^{r}\circ
\pi )(\tilde{g}_{\alpha }^{a}\circ \pi )\partial _{j}(g_{a}^{\gamma }\circ
\pi )\partial _{i}(g_{d}^{\alpha }\circ \pi ).
\end{align}%
Similarly, we have
\begin{align}\label{lili}
& U^{d}((g_{e}^{\alpha }u^{e}g_{d}^{\lambda }((\rho _{\alpha }^{i}\rho
_{\lambda }^{j})\circ h))\circ \pi )\partial _{i}(\tilde{g}_{\gamma
}^{r}\circ \pi )\partial _{j}((g_{c}^{\gamma }v^{c})\circ \pi )  \notag
\nonumber\\
& \ \ \ =-U^{d}(\tilde{g}_{\gamma }^{r}\circ \pi )((\tilde{g}_{\beta
}^{b}g_{d}^{\lambda }g_{e}^{\sigma }u^{e}((\rho _{\sigma }^{i}\rho
_{\lambda }^{j})\circ h))\circ \pi )\partial _{j}((g_{c}^{\beta
}v^{c})\circ \pi )\partial _{i}(g_{b}^{\gamma }\circ \pi ),
\end{align}
\begin{align}\label{lili1}
& U^{d}((g_{e}^{\alpha }u^{e}g_{c}^{\lambda }v^{c}g_{d}^{\mu }\partial _{i}(%
\tilde{g}_{\gamma }^{r}\circ \pi )((L_{\lambda \mu }^{\gamma }\rho _{\alpha
}^{i})\circ h))\circ \pi )  \notag \nonumber\\
& \ \ \ =-U^{d}(\tilde{g}_{\gamma }^{r}\circ \pi )((\tilde{g}_{\beta
}^{b}g_{c}^{\lambda }v^{c}g_{e}^{\sigma }u^{e}((L_{\lambda \mu
}^{\beta }\rho _{\sigma }^{i})\circ h))\circ \pi )\partial
_{i}(g_{b}^{\gamma }\circ \pi ),
\end{align}
\begin{align}\label{lili2}
& U^{d}(\tilde{g}_{\gamma
}^{r}\circ \pi )((\tilde{g}_{\beta }^{b}g_{c}^{\lambda
}v^{c}g_{e}^{\sigma }u^{e}((\rho _{\sigma }^{i}\rho _{\lambda
}^{j})\circ h))\circ \pi )\partial _{j}(g_{d}^{\beta }\circ \pi
)\partial _{i}(g_{b}^{\gamma }\circ \pi )  \notag \nonumber\\
& \ \ \ =-U^{d}((g_{e}^{\beta }v^{e}g_{c}^{\lambda }u^{c}((\rho _{\beta
}^{i}\rho _{\lambda }^{j})\circ h))\circ \pi )\partial _{i}(\tilde{g}%
_{\gamma }^{r}\circ \pi )\partial _{j}(g_{d}^{\gamma }\circ \pi
),
\end{align}
\begin{align}\label{lili3}
& U^{d}((g_{e}^{\beta }v^{e}g_{d}^{\lambda }((\rho _{\beta }^{i}\rho
_{\lambda }^{j})\circ h))\circ \pi )\partial _{i}(\tilde{g}_{\gamma
}^{r}\circ \pi )\partial _{j}((g_{c}^{\gamma }u^{c})\circ \pi )
\notag \nonumber\\
& \ \ \ =-U^{d}(\tilde{g}_{\gamma }^{r}\circ \pi )((\tilde{g}_{\alpha
}^{a}g_{d}^{\lambda }g_{e}^{\sigma }v^{e}((\rho _{\sigma }^{i}\rho _{\lambda
}^{j})\circ h))\circ \pi )\partial _{j}((g_{c}^{\alpha }u^{c})\circ \pi
)\partial _{i}(g_{a}^{\gamma }\circ \pi ),
\end{align}%
and
\begin{align}\label{lili4}
&U^{d}((g_{e}^{\beta }v^{e}g_{c}^{\lambda }u^{c}g_{d}^{\mu }\partial _{i}(%
\tilde{g}_{\gamma }^{r}\circ \pi )((L_{\lambda \mu }^{\gamma }\rho _{\beta
}^{i})\circ h))\circ \pi )  \notag  \nonumber \\
& \ \ \ =-U^{d}(\tilde{g}_{\gamma }^{r}\circ \pi )((\tilde{g}_{\alpha
}^{a}g_{c}^{\lambda }u^{c}g_{e}^{\sigma }v^{e}((L_{\lambda \mu }^{\alpha
}\rho _{\sigma }^{i})\circ h))\circ \pi )\partial _{i}(g_{a}^{\gamma }\circ
\pi ).
\end{align}%
Setting (\ref{lili})-(\ref{lili4}) in (\ref{lili0}) we deduce that $C^{r}=0$%
. This equation together with (\ref{lili00}) and (\ref{lili000}) give us (iii).
\end{proof}
\section{The generalized tangent bundle of a dual vector bundle}
We consider the following diagrams:
\begin{equation*}
\begin{array}{ccccccc}
\overset{\ast }{E} & ^{\underrightarrow{~\ \ \overset{\ast }{g}~\ \ }} & ( F,[ ,%
] _{F,h}) & ^{\underrightarrow{~\ \ \rho ~\ \ }} & TM & ^{%
\underrightarrow{~\ \ Th~\ \ }} & ( TN,[ ,] _{TN}) \\
~\downarrow \overset{\ast }{\pi } &  & \downarrow \nu &  & ~\ \ \downarrow
\tau _{M} &  & ~\ \ \downarrow \tau _{N} \ \ \ \ ,\\
M & ^{\underrightarrow{~\ \ h~\ \ }} & N & ^{\underrightarrow{~\ \
\eta ~\ \
}} & M & ^{\underrightarrow{~\ \ h~\ \ }} & N%
\end{array}%
\end{equation*}%
where $(( F,\nu ,N) ,[ ,] _{F,h},( \rho ,\eta ) )$ is a generalized
lie algebroid, $( E,\pi ,M) $ is a vector bundle and $(\overset{\ast
}{g},h) $ is a vector bundle morphism
from $( \overset{\ast }{E},\overset{\ast }{\pi },M) $ to $%
( F,\nu ,N) $ with components $g^{\alpha b},~\alpha \in
\{ 1,2,..,n\} $ and $b\in\{ 1,2,..,r\} $.

Setting $( x^{i},p_{a}) $ as the canonical local coordinates on $%
( \overset{\ast }{E},\overset{\ast }{\pi },M) ,$ where $i\in
1,\cdots ,m$, $a\in 1,\cdots ,r$, and
\begin{equation*}
( x^{i},p_{a}) \longrightarrow ( x^{i^{\prime }}( x^{i})
,p_{a^{\prime }}( x^{i},p_{a}) ) ,
\end{equation*}%
is a change of coordinates on $( \overset{\ast }{E},\overset{\ast }{\pi
},M) $, then the coordinates $p_{a}$ change to $p_{a^{\prime }}$
according to the rule
\begin{equation*}
\begin{array}{c}
p_{a^{\prime }}=\displaystyle M_{a^{\prime }}^{a}p_{a}.%
\end{array}%
\end{equation*}

If $v=v^{\alpha }t_{\alpha }$ is a section of $(F,\nu ,N)$, then we define
its corresponding section $X=X^{\alpha }\overset{\ast }{T}_{\alpha }$ in the
pull-back vector bundle $(( h\circ \overset{\ast }{\pi })
^{\ast }F,( h\circ \overset{\ast }{\pi }) ^{\ast }\nu ,\overset{%
\ast }{E}) $ given by
\begin{align*}
X(\overset{\ast }{v}_{x}) &=(( g^{\alpha b}\circ h^{-1})( u_{b}\circ
h^{-1}) t_{\alpha })( h\circ \overset{\ast }{\pi }( \overset{\ast
}{v}_{x}))g^{\alpha b}(x) u_{b}( x) t_{\alpha }( h( x)) ,\ \ \ \ \
\forall\overset{\ast }{v}_{x}\in \overset{\ast }{E}.
\end{align*}

Let $(\overset{\ast }{\partial }_{i},\dot{\partial}^{a})$ be the base
sections for the Lie $\mathcal{F}(\overset{\ast }{E})$-algebra
\begin{equation*}
(\Gamma (T\overset{\ast }{E},\tau _{\overset{\ast }{E}},\overset{\ast }{E}%
),+,\cdot ,[,]_{T\overset{\ast }{E}}).
\end{equation*}
If we define
\begin{equation*}
\begin{array}{rcl}
\ ( h\circ \overset{\ast }{\pi }) ^{\ast }F & ^{\underrightarrow{%
\overset{( h\circ \overset{\ast }{\pi }) ^{\ast }F}{\rho }}} & T%
\overset{\ast }{E}, \\
\displaystyle X^{\alpha }\overset{\ast }{T}_{\alpha }(\overset{\ast }{u}_{x})
& \longmapsto & \displaystyle( g^{\alpha b}\circ \overset{\ast }{\pi }%
) (u_{b}\circ \overset{\ast }{\pi })( \rho _{\alpha }^{i}\circ
h\circ \overset{\ast }{\pi }) \overset{\ast }{\partial }_{i}(
\overset{\ast }{u}_{x}) ,%
\end{array}%
\end{equation*}%
then $({\overset{( h\circ \overset{\ast }{\pi }) ^{\ast }F}{\rho
}},Id_{\overset{\ast }{E}})$ is a vector bundles morphism from $(
( h\circ \overset{\ast }{\pi }) ^{\ast }F,( h\circ \overset{%
\ast }{\pi }) ^{\ast }\nu ,\overset{\ast }{E}) $ to $( T%
\overset{\ast }{E},\tau _{\overset{\ast }{E}},\overset{\ast
}{E}) $. Moreover, the operation
\begin{equation*}
\begin{array}{ccc}
\Gamma(( h\circ \overset{\ast }{\pi }) ^{\ast }F,(
h\circ \overset{\ast }{\pi }) ^{\ast }\nu ,\overset{\ast }{E})
^{2} & ^{\underrightarrow{~\ \ [ ,] _{( h\circ \overset{\ast
}{\pi }) ^{\ast }F}~\ \ }} & \Gamma(( h\circ \overset{\ast
}{\pi }) ^{\ast }F,( h\circ \overset{\ast }{\pi }) ^{\ast
}\nu ,\overset{\ast }{E}) ,%
\end{array}%
\end{equation*}%
defined by
\begin{align*}
[ \overset{\ast }{T}_{\alpha },\overset{\ast }{T}_{\beta }] _{(
h\circ \overset{\ast }{\pi }) ^{\ast }F} &\!\!=(L_{\alpha \beta
}^{\gamma }\circ h\circ \overset{\ast }{\pi })\overset{\ast
}{T}_{\gamma },\\
[ \overset{\ast }{T}_{\alpha },f\overset{\ast }{T}_{\beta }]
_{( h\circ \overset{\ast }{\pi }) ^{\ast }F} & \displaystyle%
=\!\!f( L_{\alpha \beta }^{\gamma }\circ h\circ \overset{\ast }{\pi }%
) \overset{\ast }{T}_{\gamma }+( \rho _{\alpha }^{i}\circ h\circ
\overset{\ast }{\pi }) \overset{\ast }{\partial }_{i}( f)
\overset{\ast }{T}_{\beta },\\
[ f\overset{\ast }{T}_{\alpha },\overset{\ast }{T}_{\beta }]
_{( h\circ \overset{\ast }{\pi }) ^{\ast }F} & \!\!=-[ \overset{%
\ast }{T}_{\beta },f\overset{\ast }{T}_{\alpha }] _{( h\circ
\overset{\ast }{\pi }) ^{\ast }F},
\end{align*}
for any $f\in \mathcal{F}( \overset{\ast }{E}) $, is a Lie
bracket on $\Gamma(( h\circ \overset{\ast }{\pi }) ^{\ast }F,(
h\circ \overset{\ast }{\pi }) ^{\ast }\nu ,\overset{\ast }{E})$.
It is easy to check that
\begin{equation*}
((( h\circ \overset{\ast }{\pi }) ^{\ast }F,( h\circ \overset{\ast
}{\pi }) ^{\ast }\nu ,\overset{\ast }{E}),[ ,] _{( h\circ
\overset{\ast }{\pi }) ^{\ast }F},(\overset{(
h\circ \overset{\ast }{\pi }) ^{\ast }F}{\rho },Id_{\overset{\ast }{E}%
})),
\end{equation*}%
is a Lie algebroid which is called the \textit{pull-back Lie algebroid of
the generalized Lie algebroid} $(\left( F,\nu ,N\right) ,\left[ ,\right]
_{F,h},\left( \rho ,\eta \right) )$.

Now, we consider the vector subbundle $(\left( \rho ,\eta \right) T\overset{%
\ast }{E},\left( \rho ,\eta \right) \tau _{\overset{\ast }{E}},\overset{\ast
}{E})$ of the vector bundle $(\left( h\circ \pi \right) ^{\ast }F\oplus T%
\overset{\ast }{E},\overset{\oplus }{\overset{\ast }{\pi }},\overset{\ast }{E%
}),$ for which the $\mathcal{F}(\overset{\ast }{E})$-module of sections is
the $\mathcal{F}(\overset{\ast }{E})$-submodule of $(\Gamma (\left( h\circ
\pi \right) ^{\ast }F\oplus T\overset{\ast }{E},\overset{\oplus }{\overset{%
\ast }{\pi }},\overset{\ast }{E}),+,\cdot ),$ generated by the set
of sections $(\overset{\ast }{\tilde{\partial}}_{\alpha
},{\dot{\tilde{\partial}}}^{a}),$ where
\[
\overset{\ast }{\tilde{%
\partial}}_{\alpha}=\overset{\ast }{T}_{\alpha}\oplus(\rho^i_\alpha\circ h\circ\overset{\ast }{\pi})\overset{\ast
}{\partial}_{i},\ \ \
{\dot{\tilde{\partial}}}^a=0_{(h\circ\overset{\ast }{\pi})^\ast
F}\oplus{\dot{\partial}}^a.
\]
The base sections $(\overset{\ast }{\tilde{\partial}}_{\alpha },{\dot{%
\tilde{\partial}}}^{a})$ are called the \emph{natural }$(\rho ,\eta )$\emph{%
-base}.
Now consider the vector bundles morphism $( \overset{\ast }{\tilde{\rho}%
},Id_{\overset{\ast }{E}}) $ from $((\rho ,\eta )T\overset{\ast }{E}%
,(\rho ,\eta )\tau _{\overset{\ast }{E}}$ $,\overset{\ast }{E})$ to $(T%
\overset{\ast }{E},\tau _{\overset{\ast }{E}},\overset{\ast }{E})$ , where
\begin{equation*}
\begin{array}{rcl}
\left( \rho ,\eta \right) T\overset{\ast }{E} & \!\!^{\underrightarrow{\ \
\overset{\ast }{\tilde{\rho}}\ \ }}\!\!\! & \!\!T\overset{\ast }{E}\vspace*{%
2mm}, \\
(X^{\alpha }\overset{\ast }{\tilde{\partial}}_{\alpha }+\tilde{X}_{a}\dot{%
\tilde{\partial}}^{a})\!(\overset{\ast }{u}_{x})\!\!\!\! & \!\!\longmapsto
\!\!\! & \!\!(\!X^{\alpha }(\rho _{\alpha }^{i}{\circ h\circ }\overset{\ast }%
{\pi })\overset{\ast }{\partial }_{i}{+}\tilde{X}_{a}\dot{\partial}^{a})\!(%
\overset{\ast }{u}_{x}).%
\end{array}%
\end{equation*}
Moreover, we define the Lie bracket $[,]_{(\rho ,\eta )T\overset{\ast }{E}}$
as follows%
\begin{align*}
&[X_{1}^{\alpha }\overset{\ast }{\tilde{\partial}}_{\alpha }+\tilde{X}%
_{a}^{1}\dot{\tilde{\partial}}^{a},X_{2}^{\beta }\overset{\ast }{\tilde{%
\partial}}_{\beta }+\tilde{X}_{b}^{2}\dot{\tilde{\partial}}^{b}]_{\left(
\rho ,\eta \right) T\overset{\ast }{E}}=[X_{1}^{\alpha }\overset{\ast }{T}_{\alpha }+X_{2}^{\beta }\overset{%
\ast }{T}_{\beta }]_{\left( h\circ \overset{\ast }{\pi }\right)
^{\ast
}F}\nonumber\\
&\oplus [X_{1}^{\alpha }(\rho _{\alpha }^{i}\circ h\circ \overset{\ast }{%
\pi })\overset{\ast }{\partial }_{i}+\tilde{X}_{a}^{1}\dot{\partial}^{a},%
\Big.\hfill \displaystyle\Big.X_{2}^{\beta }(\rho _{\beta }^{j}\circ h\circ
\overset{\ast }{\pi })\overset{\ast }{\partial }_{j}+\tilde{X}_{b}^{2}\dot{%
\partial}^{b}]_{T\overset{\ast }{E}}.%
\end{align*}
Easily we obtain that $([ ,] _{( \rho ,\eta) T\overset%
{\ast }{E}},( \overset{\ast }{\tilde{\rho}},Id_{\overset{\ast }{E}%
}) )$ is a Lie algebroid structure for the vector bundle $(
( \rho ,\eta) T\overset{\ast }{E},( \rho ,\eta) \tau
_{\overset{\ast }{E}},\overset{\ast }{E}) $ which is called the
generalized tangent bundle.
\section{Vertical and complete $(\stackrel{\ast}{g}, h)$-lifts of sections of a dual vector bundle}
In this section, we consider the following diagrams
\begin{equation*}
\begin{array}{ccccccc}
\overset{\ast }{E} & ^{\underrightarrow{~\ \ \overset{\ast }{g}~\ \ }} &( F,[ ,%
] _{F,h}) & ^{\underrightarrow{~\ \ \rho ~\ \ }} & TM & ^{%
\underrightarrow{~\ \ Th~\ \ }} & ( TN,[ ,] _{TN}) \\
~\downarrow \overset{\ast }{\pi } &  & ~\ \ \ \downarrow \nu &  & ~\ \ \
\downarrow \tau _{M} &  & ~\ \ \ \downarrow \tau _{N} \ \ \ \ ,\\
M & ^{\underrightarrow{~\ \ h~\ \ }} & N & ^{\underrightarrow{~\ \ \eta ~\ \
}} & M & ^{\underrightarrow{~\ \ h~\ \ }} & N%
\end{array}%
\end{equation*}%
where $(( F,\nu ,N) ,[ ,] _{F,h},( \rho ,\eta)) $ is a generalized
Lie algebroid. We admit that $( \overset{\ast }{g},h) $ is a vector
bundles morphism locally invertible from $( \overset{\ast
}{E},\overset{\ast }{\pi },M) $ to $( F,\nu ,N) $ with
components
\begin{equation*}
\begin{array}{c}
g^{\alpha b},\ \ \alpha \in \{1,\cdots,n\},\ \ b\in \{1,\cdots,r\}.%
\end{array}%
\end{equation*}
So, for any vector local $( m+r) $-chart $( V,t_{V}) $
of $( \overset{\ast }{E},\overset{\ast }{\pi },M) $, there exists
the real functions
\begin{equation*}
\begin{array}{ccc}
\tilde{g}_{b\alpha }: \ V & ^{\underrightarrow{~\ \ \ \tilde{g}_{b\alpha }~\ \ }}\!\!\!\! & \mathbb{R%
},\ \ \ \alpha \in \{1,\cdots,n\},\ \ b\in \{1,\cdots,r\},%
\end{array}%
\end{equation*}%
such that%
\begin{equation*}
\begin{array}{c}
\tilde{g}_{b\alpha }\left( \varkappa \right) \cdot g^{\alpha
a}\left( \varkappa \right) =\delta _{b}^{a},\ \ \ \ %
g^{\alpha a}\left( \varkappa \right) \tilde{g}_{a\beta }\left(
\varkappa
\right) =\delta _{\beta }^{\alpha },%
\end{array}%
\end{equation*}%
for any $\varkappa \in V$. So, we can discuss about vector bundles
morphism $({\overset{\ast }{g}}^{-1},h^{-1}) $ from $( F,\nu ,N) $
to $( \overset{\ast }{E},\overset{\ast }{\pi },M) $ with components
$\tilde{g}_{b\alpha }, \alpha \in \{1,\cdots,n\}, b\in
\{1,\cdots,r\}.$

\begin{definition}
If $f\in \mathcal{F}\left( N\right) $ (respectively $f\in
\mathcal{F}\left(
M\right) ),$ then the real function $f^{\vee }=f\circ h\circ \overset{\ast }{%
\pi }$ (respectively $f^{\vee }=f\circ \overset{\ast }{\pi })$ is
called the vertical lift of the function $f.$
\end{definition}
\begin{remark}
Since
\begin{equation*}
\begin{array}{c}
(\Gamma (Th\circ \rho ,h\circ \eta )\Gamma ( \overset{\ast }{g},h)
(u_{a}s^{a}))(f)=( g^{\alpha b}u_{b}) \circ h^{-1}\rho
_{\alpha }^{i}\frac{\partial (f\circ h)}{\partial x^{i}}\circ h^{-1},%
\end{array}%
\end{equation*}%
then using the above definition we obtain
\begin{equation*}
(\Gamma (Th\circ \rho, h\circ \eta)\Gamma(\overset{\ast}{g}, h)(u_{a}s^{a})(f)%
)^{v}=(( g^{\alpha b}u_{b}\rho _{\alpha }^{i}\circ h) \circ
\overset{\ast }{\pi })\overset{\ast }{\partial _{i}}(f\circ h\circ
\overset{\ast }{\pi }).
\end{equation*}
\end{remark}
\begin{definition}
If $\overset{\ast }{u}=u_{a}s^{a}$ is a section of $( \overset{\ast }{E}%
,\overset{\ast }{\pi },M) $, then we introduce the vertical lift of $%
\overset{\ast }{u}$ as a section of $( T\overset{\ast }{E},\tau _{\overset%
{\ast }{E}},\overset{\ast }{E}) $ given by,
\begin{equation*}
\overset{\ast }{u}^{\vee }=(u_{a}\circ \pi )\dot{\partial}^{a}.
\end{equation*}
\end{definition}
If $\{s^{a}\}$ be a basis of sections of $\Gamma ( \overset{\ast }{E},%
\overset{\ast }{\pi },M) $, then using the above equation we have
\begin{equation*}
( s^{a}) ^{\vee }=\dot{\partial}^{a}.
\end{equation*}
Using the locally expression of $\overset{\ast }{u}^{\vee }$ we can
deduce
\begin{lemma}
If $\overset{\ast }{u}$ and $\overset{\ast }{v}$ are sections of $(
\overset{\ast }{E},\overset{\ast }{\pi },M) $ and $f\in \mathcal{F}(M)$%
, then
\begin{equation*}
(\overset{\ast }{u}+\overset{\ast }{v})^{\vee }=\overset{\ast }{u}^{\vee }+\overset{\ast }{v}%
^{\vee },\ \ \ (f\overset{\ast }{u})^{\vee }=f^{\vee }\overset{\ast }{u}%
^{\vee },\ \ \ \overset{\ast }{u}^{\vee }\left( f^{\vee }\right) =0.
\end{equation*}
\end{lemma}

\begin{definition}
For any $\overset{\ast }{u}\in \Gamma ( \overset{\ast }{E},\overset{%
\ast }{\pi },M) $, the $\mathcal{F}( M) $-multilinear application
\begin{equation*}
\begin{array}{rcl}
\Lambda ( \overset{\ast }{E},\overset{\ast }{\pi },M) & ^{%
\underrightarrow{~\ \ (\overset{\ast }{g}, h)
\mathcal{L}_{\overset{\ast }{u}}~\ \
}} & \Lambda ( \overset{\ast }{E},\overset{\ast }{\pi },M)%
\end{array}%
,
\end{equation*}%
defined by%
\begin{equation*}
\begin{array}{c}
(\overset{\ast }{g},h) \mathcal{L}_{\overset{\ast }{u}}(f) =(
\overset{\ast }{g},h) ^{\ast }\{ \Gamma ( Th\circ \rho ,h\circ \eta ) %
[ \Gamma (\overset{\ast }{g},h) \overset{\ast }{u}] ( \overset{\ast
}{g}^{-1},h^{-1}) ^{\ast }f\} ,~\forall f\in \mathcal{F}(
M),%
\end{array}%
\end{equation*}%
and
\begin{equation*}
\begin{array}{l}
(\overset{\ast }{g},h) \mathcal{L}_{\overset{\ast }{u}}\overset{\ast }{\omega }%
( \overset{\ast }{u}_{1},\cdots,\overset{\ast }{u}_{q})
=(\overset{\ast }{g},h) ^{\ast }\{ \mathcal{L}_{\Gamma ( g,h)
\overset{\ast }{u}}(\overset{\ast }{g}^{-1},h^{-1}) ^{\ast }\overset{\ast }{%
\omega }( \Gamma (\overset{\ast }{g},h) \overset{\ast
}{u}_{1},\cdots,\Gamma
(\overset{\ast }{g},h) \overset{\ast }{u}_{q}) \} \\
=(\overset{\ast }{g},h) ^{\ast }\{ \Gamma ( Th\circ \rho ,h\circ
\eta ) [ \Gamma (\overset{\ast }{g},h) \overset{\ast }{u}]
(\overset{\ast }{g}^{-1},h^{-1}) ^{\ast }\overset{\ast }{\omega }(
\Gamma (\overset{\ast }{g},h) \overset{\ast }{u}_{1},\cdots,\Gamma
(\overset{\ast }{g},h) \overset{\ast
}{u}_{q}) \} \\
\ \ \ -(\overset{\ast }{g},h) ^{\ast }\{ (\overset{\ast }{g}^{-1},h^{-1}) ^{\ast }%
\overset{\ast }{\omega }( \Gamma (\overset{\ast }{g},h) \overset{\ast }{u}%
_{1},\cdots,[ \Gamma (\overset{\ast }{g},h) \overset{\ast
}{u},\Gamma (\overset{\ast }{g},h) \overset{\ast }{u}_{i}]
_{F,h},\cdots,\Gamma (\overset{\ast }{g},h)
\overset{\ast }{u}_{q}) \} ,%
\end{array}%
\end{equation*}%
for any $\omega \in \Lambda ^{q}\mathbf{\ }( \overset{\ast }{E},\overset%
{\ast }{\pi },M) $ and $\overset{\ast }{u}_{1},...,\overset{\ast }{u}%
_{q}\in \Gamma ( \overset{\ast }{E},\overset{\ast }{\pi },M) ,$
will be called \emph{the covariant Lie }$( g,h) $-\emph{%
derivative with respect to the section }$\overset{\ast }{u}.$
\end{definition}

\begin{definition}
For any $a=1,\cdots ,r,$ we consider the real function $U_{a}$ on $\overset{%
\ast }{E}$ such that
\begin{equation*}
U_{a}|_{\overset{\ast }{\pi }^{-1}(V)}( \overset{\ast }{u}_{x})
=p_{a},
\end{equation*}%
where the real numbers $p_{1},\cdots ,p_{r}$ are the fibre components of the
point $\overset{\ast }{u}_{x}$ in the arbitrary vector local $(m+r)$-chart $({V%
},s_{{V}}).$
\end{definition}

\begin{remark}
Using the above definition we have $ \dot{\partial}^{b}\left(
U_{a}\right) =\delta _{a}^{b}$ and $\partial _{i}\left( U_{a}\right)
=0, $ where $\alpha \in \{1,\cdots,n\}, b\in \{1,\cdots,r\}.$
\end{remark}

\begin{definition}
If $\overset{\ast }{\omega }=\omega ^{a}s_{a}\in \Lambda ^{1}\mathbf{\ }%
( \overset{\ast }{E},\overset{\ast }{\pi },M) ,$ then we consider
the real function $\widehat{\overset{\ast }{\omega }}$ defined by
\begin{equation*}
\widehat{\overset{\ast }{\omega }}|_{\overset{\ast }{\pi }%
^{-1}(V)}=U_{a}( \omega ^{a}\circ \overset{\ast }{\pi }) |_{%
\overset{\ast }{\pi }^{-1}(V)},
\end{equation*}%
where $({V},s_{{V}})$ is an arbitrary vector local $(m+r)$-chart$.$
\end{definition}

\begin{theorem}
Let $\overset{\ast }{u}$ be a section of $( \overset{\ast }{E},\overset{%
\ast }{\pi },M) $. Then there exists a unique vector field $\overset{%
\ast }{u}^{c}\in \Gamma ( T\overset{\ast }{E},\tau _{\overset{\ast }{E}%
},\overset{\ast }{E}) $, the complete $(\overset{\ast}{g}, h)$-lift of $\overset{\ast }{u}$,
satisfying the following conditions:

\textbf{i}) $\overset{\ast }{u}^{c}$ is $(h\circ \overset{\ast }{\pi })$%
-related, i.e.,
\begin{equation*}
T(h\circ \overset{\ast }{\pi })(\overset{\ast }{u}_{\overset{\ast }{v}%
_{x}}^{c})=\{ \Gamma (Th\circ \rho ,h\circ \eta )(\Gamma
(\overset{\ast }{g},h) ( \overset{\ast }{u}) )\}(h\circ
\overset{\ast }{\pi }( \overset{\ast }{v}_{x})),
\end{equation*}

\textbf{ii}) $\overset{\ast }{u}^{c}(\widehat{\overset{\ast }{\omega }})=%
\widehat{(\overset{\ast }{g},h) \mathcal{L}_{\overset{\ast }{u}}\overset{\ast }{%
\omega }}$,\newline
for any $\overset{\ast }{\omega }\in \Lambda ^{1}\mathbf{\ }( \overset{%
\ast }{E},\overset{\ast }{\pi },M) $.
\end{theorem}
\begin{proof}
Similar to the proof of Theorem \ref{T1}, we obtain the following
locally expression for $\overset{\ast }{u}^{c}$  that show the
existence and uniqueness of it.
\begin{equation*}
\begin{array}{cl}
\overset{\ast }{u}^{c} & =(( g^{\alpha e}u_{e}\rho _{\alpha
}^{i}\circ h) \circ \overset{\ast }{\pi })\overset{\ast }{%
\partial }_{i}-U_{a}(K^{\gamma a}( \overset{\ast }{u})
\circ h\circ \overset{\ast }{\pi })( \tilde{g}_{b\gamma }\circ
\overset{\ast }{\pi }){\dot{\partial}}^{b},
\end{array}%
\end{equation*}
where
\[
K^{\gamma a}( \overset{\ast }{u}) =( g^{\beta
c}u_{c}\rho _{\beta }^{j}\frac{\partial \left( g^{\gamma a}\right) }{%
\partial x^{i}}-g^{\alpha a}\rho _{\alpha }^{i}\frac{\partial \left(
g^{\gamma c}u_{c}\right) }{\partial x^{i}}+g^{\alpha c}u_{c}L_{\alpha \beta
}^{\gamma }g^{\beta a})\circ h^{-1}.
\]

\end{proof}
\begin{definition}
The complete $( \overset{\ast}{g},h) $-lift of a function $f\in\mathcal{F}(N)$ into $\mathcal{F}(\overset{\ast }{E})$ is the function
\begin{equation*}
f^{c}:\overset{\ast }{E}\longrightarrow \mathbb{R},
\end{equation*}%
defined by
\begin{equation*}
f^{c}|_{\overset{\ast }{\pi }^{-1}(V)}=U_{a}( g^{\alpha a}\circ \overset%
{\ast }{\pi }) (\rho _{\alpha }^{i}\circ h\circ \overset{\ast }{\pi })%
\overset{\ast }{\partial }_{i}(f\circ h\circ \overset{\ast }{\pi })|_{%
\overset{\ast }{\pi }^{-1}(V)},
\end{equation*}%
where $({V},s_{{V}})$ is an arbitrary vector local $(m+r)$-chart.
\end{definition}
Similar to the Lemmas \ref{12} and \ref{13}, we have
\begin{lemma}
If $\overset{\ast }{u}$ is a section of $( \overset{\ast }{E},\overset{%
\ast }{\pi },M) $ and $f, f_1, f_2\in \mathcal{F}(N)$, then%
\begin{equation*}
\begin{array}{l}
(i)\ (f_1+f_2)^{c}=f_1^{c}+f_2^{c}, \\
(ii)\ (f_1f_2)^{c}=f_1^{c}f_2^{\vee }+f_1^{\vee }f_2^{c}, \\
(iii)\ \overset{\ast }{u}^{\vee }( f^{c}) =\{ \Gamma
(Th\circ \rho ,h\circ \eta )(\Gamma (\overset{\ast }{g},h) \overset{\ast }{u%
})( f) \} ^{\vee },\\
(iv)\ \overset{\ast }{u}^{c}( f^{c}) =\{ (\Gamma (Th\circ \rho
,h\circ \eta )(\Gamma (\overset{\ast }{g},h) \overset{\ast }{u})
(f)\} ^{c}.
\end{array}%
\end{equation*}
\end{lemma}
\begin{definition}
The complete $(\overset{\ast }{g},h) $-lift $\overset{\ast }{u}^{C}$
of a section $\overset{\ast }{u}\in \Gamma (\overset{\ast
}{E},\overset{\ast }{\pi },M)$ is the section of $( (\rho ,\eta
)T\overset{\ast }{E},(\rho
,\eta )\tau _{\overset{\ast }{E}},\overset{\ast }{E}) $ given by%
\begin{align}\label{L5}
\overset{\ast }{u}^{C}&=(( g^{\alpha e}\circ \overset{\ast }{\pi })
(u_{e}\circ \overset{\ast }{\pi }))\tilde{\partial}_{\alpha }-U_{a}(
K^{\gamma a}( \overset{\ast }{u}) \circ h\circ \overset{\ast }{%
\pi })( \tilde{g}_{b\gamma }\circ \overset{\ast }{\pi })
\dot{\tilde{\partial}}^{b}.%
\end{align}
\end{definition}
It is easy to check that $\Gamma( \overset{\ast }{\tilde{\rho}},Id_{\overset{\ast }{E}}) (%
\overset{\ast }{u}^{C})=\overset{\ast }{u}^{c}$.
\begin{definition}
If $\overset{\ast }{u}=u_{a}s^{a}$ be a section of $( \overset{\ast }{E}%
,\overset{\ast }{\pi },M) $, then we introduce the vertical lift of $%
\overset{\ast }{u}$ as section of $( \left( \rho ,\eta \right) T\overset%
{\ast }{E},\left( \rho ,\eta \right) \tau _{\overset{\ast }{E}},\overset{%
\ast }{E}) $ given by
\begin{equation*}
\overset{\ast }{u}^{V}=0_{(h\circ \overset{\ast }{\pi })^{\ast }E}\oplus
\overset{\ast }{u}^{\vee }.
\end{equation*}
\end{definition}
If $u=u^{a}e_{a}\in \Gamma (E,\pi ,M)$, then in the locally
expressions we get
\begin{align}\label{L6}
\overset{\ast }{u}^{V}=(u_{a}\circ \overset{\ast }{\pi
})\dot{\tilde{\partial}}^{a},
\end{align}
which gives us
$
\left( s^{a}\right) ^{V}=\dot{\tilde{\partial}}^{a}.
$
\begin{remark}
Using the almost tangent $(\overset{\ast }{g},h) $-structure $\overset{\ast }{\mathcal{J}}_{(\overset{\ast }{g},h) }$ given by\textrm{\ }%
\begin{equation*}
\begin{array}{rcl}
\Gamma (( \rho ,\eta) T\overset{\ast }{E},( \rho
,\eta) \tau _{\overset{\ast }{E}},\overset{\ast }{E}) & ^{%
\underrightarrow{\overset{\ast }{\mathcal{J}}_{(\overset{\ast
}{g},h) }}} & \Gamma ( ( \rho ,\eta) T\overset{\ast }{E},( \rho
,\eta ) \tau _{\overset{\ast }{E}},\overset{\ast }{E}) \\
Z^{\alpha }\tilde{\partial}_{\alpha }+Y^{b}{\dot{\tilde{\partial}}%
_{b}}& \longmapsto &( \tilde{g}_{b\alpha }\circ \overset{\ast }{\pi }%
) Z^{\alpha }{\dot{\tilde{\partial}}}^{b}%
\end{array}%
\end{equation*}%
it results that $\overset{\ast }{\mathcal{J}}_{(\overset{\ast
}{g},h) }( \overset{\ast }{u}^{C}) =\overset{\ast }{u}^{V}.$
\end{remark}
Similar to Theorem \ref{KM} we can deduce the following
\begin{theorem}
The Lie brackets of generalized vertical lifts and generalized complete $%
( \overset{\ast}{g}, h) $-lifts satisfy the following equalities:
\begin{align*}
i)\ [\overset{\ast }{u}^{V},\overset{\ast }{v}^{V}]_{(\rho ,\eta )T\overset{%
\ast }{E}} &=0,\ \  \\
ii)\ [\overset{\ast }{u}^{V},\overset{\ast }{v}^{C}]_{(\rho ,\eta )T\overset{%
\ast }{E}} &=\{ \Gamma (\overset{\ast }{g}^{-1},h^{-1}) [\Gamma (
\overset{\ast }{g},h) \overset{\ast }{u},\Gamma (\overset{\ast }{g},h) \overset{\ast }{v}%
]_{F,h}\} ^{V}, \\
iii)\ [\overset{\ast }{u}^{C},\overset{\ast }{v}^{C}]_{(\rho ,\eta )T\overset%
{\ast }{E}} &=\{ \Gamma (\overset{\ast }{g}^{-1},h^{-1}) [\Gamma (
\overset{\ast }{g},h) \overset{\ast }{u},\Gamma (\overset{\ast }{g},h) \overset{\ast }{v}%
]_{F,h}\} ^{C}.
\end{align*}
\end{theorem}

\section{Legendre transformation}
\begin{definition}
A Lagrange fundamental function on the vector bundle $\left( E,\pi
,M\right) $ is a function $E~\ ^{\underrightarrow{\ \ L\ \ }}~\
\mathbb{R}$ which satisfies the following conditions:\medskip
\medskip

$L_{1}.$ $L\circ u\in C^{\infty }\left( M\right) $, for any $u\in \Gamma
\left( E,\pi ,M\right) \setminus \left\{ 0\right\} $;\smallskip

$L_{2}.$ $L\circ 0\in C^{0}\left( M\right) $, where $0$ means the null
section of $\left( E,\pi ,M\right) .$\medskip
\end{definition}

\begin{remark}
If $\left( U,s_{U}\right) $ is a local vector $\left( m+r\right) $-chart,
then we obtain the following real functions defined on $\pi ^{-1}\left(
U\right) $:%
\begin{equation*}
\begin{array}{cc}
L_{i}{=}\frac{\partial L}{\partial x^{i}}, & L_{ib}{%
=}\frac{\partial ^{2}L}{\partial x^{i}\partial y^{b}}\vspace*{2mm}, \\
L_{a}{=}\frac{\partial L}{\partial y^{a}}, & L_{ab}{%
=}\frac{\partial ^{2}L}{\partial y^{a}\partial y^{b}}.%
\end{array}%
\end{equation*}
\end{remark}
\begin{definition}
If $L$ is a Lagrange fundamental function such that
\begin{equation*}
rank\left\Vert L_{ab}\left( u_{x}\right) \right\Vert =r,
\end{equation*}%
for any $u_{x}\in \pi ^{-1}\left( U\right) \backslash \left\{ 0_{x}\right\} $%
, then we will say that \emph{the Lagrange fundamental function }$L$\emph{\
is regular }and we obtain the real functions $\tilde{L}^{ab}$\ locally
defined by%
\begin{equation*}
\begin{array}{ccc}
\pi ^{-1}\left( U\right) & ^{\underrightarrow{\ \ \tilde{L}^{ab}\ \ }} &
\mathbb{R}, \\
u_{x} & \longmapsto & \tilde{L}^{ab}\left( u_{x}\right),%
\end{array}%
\end{equation*}%
where $\left\Vert \tilde{L}^{ab}\left( u_{x}\right) \right\Vert =\left\Vert
L_{ba}\left( u_{x}\right) \right\Vert ^{-1}$, for any $u_{x}\in \pi
^{-1}\left( U\right) \backslash \left\{ 0_{x}\right\} .$
\end{definition}
\begin{definition}\label{L2}
If $L$ is a Lagrange fundamental function, then we build the
Legendre
bundles morphism%
\begin{equation*}
\begin{array}{rcl}
E & ^{\underrightarrow{~\ \ \varphi _{L}~\ \ }} & \overset{\ast }{E} \\
\pi \downarrow &  & \downarrow \overset{\ast }{\pi } \\
M & ^{\underrightarrow{~\ \ Id_{M}~\ \ }} & M%
\end{array}%
,
\end{equation*}%
locally defined%
\begin{equation}\label{L1}
\begin{array}{ccc}
\pi ^{-1}( U) & ^{\underrightarrow{~\ \ \ \varphi _{L}~\ \ }} &
\overset{\ast }{\pi }^{-1}( U), \\
u_{x}=u^{a}( x) s_{a}( x) & \longmapsto & u^{a}(
x) L_{ab}( u_{x}) s^{b}( x),%
\end{array}
\end{equation}%
for any vector local $( m+r) $-charts $( U,s_{U}) $ and $(
U,\overset{\ast }{s}_{U}) $ of $( E,\pi ,M) $ and $( \overset{\ast
}{E},\overset{\ast }{\pi },M) $ respectively.
\end{definition}
Using the above definition, we deduce that if $u=u^as_a$ belongs to $\Gamma(E, \pi, M)$, then we obtain its Legendre transformation 
\[
\Gamma(\varphi_L, Id_M)(u)= (u^a(L_{ab}\circ u))s^b,
\]
belongs to $\Gamma(\overset{\ast }{E},\overset{\ast }{\pi },M)$.
\begin{definition}
If $L$ is a Lagrange fundamental function positively homogenous of
degree two, namely

$F_{1}.$ $L$ is positively $2$-homogenous on the fibres of vector bundle $%
\left( E,\pi ,M\right) ;$\smallskip

$F_{2}.$ For any vector local $(m+r)$-chart $\left( U,s_{U}\right) $ of $%
\left( E,\pi ,M\right) ,$ the hessian%
\begin{equation*}
\left\Vert L_{~ab}\left( u_{x}\right) \right\Vert,
\end{equation*}%
is positively define for any $u_{x}\in \pi ^{-1}\left( U\right) \backslash
\left\{ 0_{x}\right\} $, then $L$ will be called Finsler fundamental
function.
\end{definition}

\begin{proposition}\label{TT}
If $L$ is a Finsler fundamental function on the vector bundle
$\left( E,\pi ,M\right) $, then
\begin{equation*}
\varphi _{L}\left( u_{x}\right) =L_{b}\left( u_{x}\right) s^{b}\left(
x\right) ,~\forall u_{x}\in E.
\end{equation*}
\end{proposition}
\begin{proof}
From (\ref{L1}) we have $\varphi _{L}( u_{x}) =u^a(x)L_{ab}( u_{x})
s^{b}(x).$ But, the Finsler fundamental function $L$ satisfies
\[
u^a(x)L_{ab}( u_{x})=L_b,
\]
because $L$ is positively $2$-homogenous on the fibres of $(E, \pi,
M)$. This complete the proof.
\end{proof}
\begin{definition}
A Hamilton fundamental function on the dual vector bundle $( \overset{%
\ast }{E},\overset{\ast }{\pi },M) $ is a function $\overset{%
\ast }{E}~\ ^{\underrightarrow{\ \ H\ \ }}~\ \mathbb{R}$ which satisfies the
following conditions:\medskip \medskip

$H_{1}.$ $H\circ \overset{\ast }{u}\in C^{\infty }( M) $, for any
$\overset{\ast }{u}\in \Gamma ( \overset{\ast }{E},\overset{\ast }{\pi }%
,M) \setminus \{ 0\} $;\smallskip

$H_{2}.$ $H\circ 0\in C^{0}( M) $, where $0$ means the null
section of $( \overset{\ast }{E},\overset{\ast }{\pi },M) .$%
\medskip
\end{definition}
\begin{remark}
If $( U,\overset{\ast }{s}_{U}) $ is a local vector $(
m+r) $-chart, then we obtain the following real functions defined on $%
\overset{\ast }{\pi }^{-1}( U) $:%
\begin{equation*}
\begin{array}{cc}
H_{i}{=}\frac{\partial H}{\partial x^{i}}, & H_{i}^{b}{=}\frac{\partial ^{2}H}{\partial x^{i}\partial p_{b}}, \\
H^{a}{=}\frac{\partial H}{\partial p_{a}}, & H^{ab}{%
=}\frac{\partial ^{2}H}{\partial p_{a}\partial p_{b}}.%
\end{array}%
\end{equation*}
\end{remark}
\begin{definition}
If $H$ is a Hamilton fundamental function such that
\begin{equation*}
\begin{array}{c}
rank\Vert H^{ab}( \overset{\ast }{u}_{x}) \Vert =r,%
\end{array}%
\end{equation*}%
for any $\overset{\ast }{u}_{x}\in \overset{\ast }{\pi }^{-1}( U)
\backslash \{ 0_{x}\} $, then we will say that \emph{the Hamilton
fundamental function }$H$\emph{\ is regular }and we obtain the real
functions $\tilde{H}_{ab}$\ locally defined by%
\begin{equation*}
\begin{array}{ccc}
\overset{\ast }{\pi }^{-1}( U) & ^{\underrightarrow{\ \ \tilde{H}%
_{ab}\ \ }} & \mathbb{R}, \\
\overset{\ast }{u}_{x} & \longmapsto & \tilde{H}_{ab}( \overset{\ast }{u%
}_{x}),%
\end{array}%
\end{equation*}%
where $\Vert \tilde{H}_{ab}( \overset{\ast }{u}_{x}) \Vert =\Vert
H^{ab}( \overset{\ast }{u}_{x})
\Vert ^{-1}$, for any $\overset{\ast }{u}_{x}\in \overset{\ast }{\pi }%
^{-1}( U) \backslash \{ 0_{x}\} .$
\end{definition}
\begin{definition}\label{L3}
If $H$ is a Hamilton fundamental function on the vector bundle $(
\overset{\ast }{E},\overset{\ast }{\pi },M) $, then we build the
Legendre bundles morphism%
\begin{equation*}
\begin{array}{rcl}
\overset{\ast }{E} & ^{\underrightarrow{~\ \ \ \varphi _{H}~\ \ }} & E \\
\overset{\ast }{\pi }\downarrow &  & \downarrow \pi \\
M & ^{\underrightarrow{~\ \ Id_{M}~\ \ }} & M%
\end{array}%
,
\end{equation*}%
where $\ \varphi _{H}$ is locally defined%
\begin{equation*}
\begin{array}{ccc}
\overset{\ast }{\pi }^{-1}( U) & ^{\underrightarrow{~\ \ \
\varphi _{H}~\ \ }} & \pi ^{-1}( U) \\
\overset{\ast }{u}_{x}=u_{a}( x) s^{a}( x), &
\longmapsto & u_{a}( x) H^{ab}( \overset{\ast }{u}%
_{x}) s_{b}( x),%
\end{array}%
\end{equation*}%
for any vector local $( m+r) $-chart $( U,s_{U}) $ of $%
( E,\pi ,M) $ and for any vector local $( m+r) $-chart
$( U,\overset{\ast }{s}_{U}) $ of $( \overset{\ast }{E},%
\overset{\ast }{\pi },M)$.
\end{definition}
Using the above definition, we deduce that if $\overset{\ast }{u}=u_{a}s^{a}$ belongs to $\Gamma( \overset{\ast }{E},\overset{%
\ast }{\pi },M)$, then we obtain its Legendre transformation 
\[
\Gamma(\varphi_H, Id_M)(\overset{\ast}{u})=( u_{a}( H^{ab}\circ u) ) s_{b},
\]
belongs to $\Gamma(E, \pi ,M)$.
\begin{definition}
If $H$ is Hamilton fundamental function positively homogeneous of
degree two, namely

$C_{1}.$ $H$ is positively $2$-homogeneous on the fibres of vector bundle $%
( \overset{\ast }{E},\overset{\ast }{\pi },M) ;$\smallskip

$C_{2}.$ For any vector local $(m+r)$-chart $( U,\overset{\ast }{s}%
_{U}) $ of $( \overset{\ast }{E},\overset{\ast }{\pi },M) ,$
the hessian:%
\begin{equation*}
\Vert H_{~}^{ab}( \overset{\ast }{u}_{x}) \Vert,
\end{equation*}%
is positively define for any $\overset{\ast }{u}_{x}\in \overset{\ast }{\pi }%
^{-1}( U) \backslash \{ 0_{x}\} $, then $H$ will be called Cartan
fundamental function.\\
Similar to proposition (\ref{TT}), we have the following
\end{definition}
\begin{proposition}
If $H$ is a Cartan fundamental function, then
\begin{equation*}
\varphi _{H}( \overset{\ast }{u}_{x}) =H^{b}( \overset{\ast }%
{u}_{x}) s_{b}( x) ,~\forall \overset{\ast }{u}_{x}\in \overset{\ast
}{E}.
\end{equation*}
\end{proposition}
\begin{theorem}
If $L$ is a Lagrange fundamental function on the vector bundle $(
E,\pi ,M) $ and $H$ is a Hamiltonian on the dual vector bundle $(
\overset{\ast }{E},\overset{\ast }{\pi },M) ,$ then:

i) $\varphi _{H}\circ \varphi _{L}=Id_{\pi ^{-1}( U) }$ if and only
if $L$ is regular and $\tilde{L}^{ab}=H^{ab}\circ \varphi _{L};$

ii) $\varphi _{L}\circ \varphi _{H}=Id_{\overset{\ast }{\pi }^{-1}(
U) }$ if and only if $H$ is regular and $\tilde{H}_{ab}=L_{ab}\circ
\varphi _{H}.$
\end{theorem}
\begin{proof}
Using definition (\ref{L2}) and (\ref{L3}), we deduce that
\begin{equation*}
\begin{array}{cl}
\varphi _{H}\circ \varphi _{L}\left( u_{x}\right) & =\varphi _{H}\left(
u^{a}\left( x\right) L_{ab}\left( u_{x}\right) s^{b}\left( x\right) \right)
\\
& =u^{a}\left( x\right) L_{ab}\left( u_{x}\right) H^{bc}\left( \varphi
_{L}\left( u_{x}\right) \right) s_{c}\left( x\right) \\
& =Id_{\pi ^{-1}\left( U\right) }\left( u_{x}\right),
\end{array}%
\end{equation*}%
if and only if
\begin{equation*}
L_{ab}\left( u_{x}\right) H^{bc}\left( \varphi _{L}\left( u_{x}\right)
\right) =\delta _{a}^{c}\left( u_{x}\right) ,
\end{equation*}%
$~$for any $u_{x}\in \pi ^{-1}\left( U\right)$. Thus we have (i).
Similar, we can prove (ii).
\end{proof}

\begin{definition}
If $L$ is a Lagrange fundamental function on the vector bundle $\left( E,\pi
,M\right) ,$ then the Hamilton fundamental function $H,$ locally defined by%
\begin{equation*}
\begin{array}{ccc}
\overset{\ast }{\pi }^{-1}\left( U\right) & ^{\underrightarrow{~\ \
H~\ \ }}
& \mathbb{R}, \\
\overset{\ast }{u}_{x}=u_{a}\left( x\right) s^{a}\left( x\right) &
\longmapsto & u_{a}\left( x\right) u^{a}\left( x\right) -L\left( u_{x}\right),%
\end{array}%
\end{equation*}%
for any vector local $( m+r) $-chart $( U,\overset{\ast }{s}%
_{U}) $ of $( \overset{\ast }{E},\overset{\ast }{\pi },M) ,$ where
$u^{a}( x)$, $a\in\{1, \cdots, r\}$, are the
components of the solution of the system of differentiable equations%
\begin{equation*}
\begin{array}{ccc}
\left\{
\begin{array}{ccc}
u_{1}\left( x\right)&=&u^{a}\left( x\right) L_{a1}\left( u_{x}\right), \\
\vdots&\vdots&\vdots\\
u_{r}\left( x\right)&=&u^{a}\left( x\right) L_{ar}\left( u_{x}\right),%
\end{array}%
\right. ,~u_{x}\in \pi ^{-1}\left( U\right) ,%
\end{array}
\end{equation*}%
will be called the \emph{Legendre transformation of the Lagrangian} $L.$
\end{definition}
It is remarkable that in the general case, if $L$ is a Lagrange fundamental function on the vector bundle $%
\left( E,\pi ,M\right) $ and $H$ is its Legendre transformation, then $%
H\circ \varphi _{L}\neq L$, but in particular, if $L$ is a Finsler
fundamental function on the vector bundle
$\left( E,\pi ,M\right) $ and $H$ is its Legendre transformation, then $%
H\circ \varphi _{L}=L.$
\begin{definition}
If $H$ is a Hamilton fundamental function on the dual vector bundle
$( \overset{\ast }{E},\overset{\ast }{\pi },M) ,$ then the Lagrange
fundamental function $L,$ locally defined by%
\begin{equation*}
\begin{array}{ccc}
\pi ^{-1}( U) & ^{\underrightarrow{~\ \ L~\ \ }} & \mathbb{R}, \\
u_{x}=u^{a}( x) s_{a}( x) & \longmapsto & u^{a}(
x) u_{a}( x) -H( \overset{\ast }{u}_{x}),%
\end{array}%
,
\end{equation*}%
for any vector local $\left( m+r\right) $-chart $\left( U,s_{U}\right) $ of $%
\left( E,\pi ,M\right) ,$ where $\left( u_{a}\left( x\right) ,~a\in
\overline{1,r}\right) $ are the components of the solution of the system of
differentiable equations%
\begin{equation*}
\begin{array}{ccc}
\left\{
\begin{array}{ccc}
u^{1}( x)&=&u_{a}( x) H^{a1}( \overset{\ast }{u}%
_{x}) \\
\vdots&\vdots&\vdots\\
u^{r}( x)&=&u_{a}( x) H^{ar}( \overset{\ast }{u}%
_{x})%
\end{array}%
\right. ,~\overset{\ast }{u}_{x}\in \overset{\ast }{\pi }^{-1}(
U) ,%
\end{array}%
\end{equation*}%
will be called the Legendre transformation of the Hamiltonian $H.$
\end{definition}
In general, if $H$ is a Hamilton fundamental function on the vector bundle $%
(\overset{\ast }{E},\overset{\ast }{\pi },M) $ and $L$ is its
Legendre transformation, then $L\circ \varphi _{H}\neq H$, but
in particular, if $H$ is a Cartan fundamental function on the vector bundle $%
( \overset{\ast }{E},\overset{\ast }{\pi },M) $ and $L$ is its
Legendre transformation, then $L\circ \varphi _{H}=H.$
\begin{remark}
The Hamilton fundamental function $H$ is the Legendre transformation of the
Lagrange fundamental function $L$ if and only if the Lagrange fundamental
function $L$ is the Legendre transformation of the Hamilton fundamental
function $H.$
\end{remark}
\section{Duality between vertical and complete lifts}
Let $L$ be a Lagrangian on the vector bundle $\left( E,\pi ,M\right)
$ and let $H$ be its Legendre transformation.

Using the Legendre bundles morphism $\left( \varphi _{L},Id_{M}\right) $, we
build the vector bundles morphism $\left( \left( \rho ,\eta \right) T\varphi
_{L},\varphi _{L}\right) $ given by the diagram%
\begin{equation*}
\begin{array}{rcl}
\left( \rho ,\eta \right) TE & ^{\underrightarrow{~\ \left( \rho ,\eta
\right) T\varphi _{L}~\ }} & \left( \rho ,\eta \right) T\overset{\ast }{E}
\\
\left( \rho ,\eta \right) \tau _{E}\downarrow &  & \downarrow \left( \rho
,\eta \right) \tau _{\overset{\ast }{E}} \\
E & ^{\underrightarrow{~\ \ \ \ \varphi _{L~\ \ \ \ }}} & \overset{\ast }{E}%
\end{array}%
,  \label{eq108}
\end{equation*}%
such that%
\begin{equation*}
\begin{array}{cl}
\Gamma ( ( \rho ,\eta ) T\varphi _{L},\varphi _{L}) ( Z^{\alpha
}\tilde{\partial}_{\alpha }) &\!\!\!\!=( Z^{\alpha }\circ \varphi
_{H}) {\dot{{\tilde{\partial}}}}_{\alpha }+[
( \rho _{\alpha }^{i}{\circ }h{\circ }\pi ) Z^{\alpha }L_{ib}%
] \circ \varphi _{H}{\dot{{\tilde{\partial}}}}^{b}, \\
\Gamma ( ( \rho ,\eta ) T\varphi _{L},\varphi _{L}) (
Y^{a}{\dot{{\tilde{\partial}}}}_{a}) &\!\!\!\!=(
Y^{a}L_{ab}) \circ \varphi _{H}{\dot{{\tilde{\partial}}}^{b}},%
\end{array}
\label{eq109}
\end{equation*}%
for any $Z^{\alpha }\tilde{\partial}_{\alpha }+Y^{a}{\dot{{\tilde{%
\partial}}}}_{a}\in \Gamma \left( \left( \rho ,\eta \right) TE,\left( \rho
,\eta \right) \tau _{E},E\right)$. The vector bundles morphism
$\left( \left( \rho ,\eta \right) T\varphi _{L},\varphi _{L}\right)
$ will be called the \emph{tangent }$\left( \rho ,\eta \right)
$-\emph{application of the Legendre bundles morphism associated to
the Lagrangian }$L$. Using this application together with
(\ref{com}), (\ref{L4}) we deduce the
following theorems.
\begin{theorem}
If $u=u^{a}s_{a}\in \Gamma \left( E,\pi ,M\right) $ such that
\begin{equation*}
\begin{array}{c}
\Gamma ( ( \rho ,\eta ) T\varphi _{L},\varphi _{L})
( u^{V}) =\Gamma(\varphi_L, Id_M)(u),%
\end{array}%
\end{equation*}%
then
\begin{align*}
u^{a}\circ \pi \circ \varphi _{H}=u^{a}\circ \overset{\ast }{\pi },\ \ \ 
L_{ab}\circ \varphi _{H}=L_{ab}\circ u\circ \overset{\ast }{\pi }.
\end{align*}
\end{theorem}
\begin{theorem}
If $u=u^{a}s_{a}\in \Gamma \left( E,\pi ,M\right) $ such that
\begin{equation*}
\begin{array}{c}
\Gamma ( ( \rho ,\eta ) T\varphi _{L},\varphi _{L})
( u^{C}) =\Gamma(\varphi_L, Id_M)(u),%
\end{array}%
\end{equation*}%
then%
\begin{equation*}
\begin{array}{rl}
( g^{\alpha e}u_{e}) \circ \overset{\ast }{\pi } &\!\!\!\!=(
g_{e}^{\alpha }u^{e}) \circ \pi \circ \varphi _{H}, \\
U_{a}[ K^{\gamma a}( \overset{\ast }{u}) \circ h\circ
\overset{\ast }{\pi }] ( \tilde{g}_{b\gamma }\circ \overset{\ast }%
{\pi }) &\!\!\!\!=\{ U^{a}( K_{a}^{\gamma }( u) \circ h\circ \pi) (
\tilde{g}_{\gamma }^{c}\circ \pi )
L_{cb}\} \circ \varphi _{H} \\
& -\{ ( \rho _{\alpha }^{i}\circ h\circ \pi )( ( g_{e}^{\alpha
}u^{e}) \circ \pi ) L_{ib}\} \circ \varphi
_{H}.%
\end{array}%
\end{equation*}
\end{theorem}
Using the bundles morphism $\left( \varphi _{H},Id_{M}\right) $, we
build the vector bundles morphism $\left( \left( \rho ,\eta \right)
T\varphi
_{H},\varphi _{H}\right) $ given by the diagram%
\begin{equation*}
\begin{array}{rcl}
\left( \rho ,\eta \right) T\overset{\ast }{E} & ^{\underrightarrow{~\ \left(
\rho ,\eta \right) T\varphi _{H}~\ }} & \left( \rho ,\eta \right) TE \\
\left( \rho ,\eta \right) \tau _{\overset{\ast }{E}}\downarrow &  &
\downarrow \left( \rho ,\eta \right) \tau _{E} \\
E^{\ast } & ^{\underrightarrow{~\ \ \ \ \varphi _{H~\ \ \ \ }}} & E%
\end{array}
, \label{eq110}
\end{equation*}%
such that%
\begin{equation*}
\begin{array}{cl}
\Gamma ( ( \rho ,\eta ) T\varphi _{H},\varphi _{H}) ( Z^{\alpha
}\overset{\ast }{\tilde{\partial}}_{\alpha }) &
\!\!\!\!=( Z^{\alpha }\circ \varphi _{L}) \tilde{\partial}_{\alpha }+%
[ ( \rho _{\alpha }^{i}{\circ }h{\circ }\overset{\ast }{\pi }%
) Z^{\alpha }H_{i}^{b}] \circ \varphi _{L}{\dot{{\tilde{\partial}}}}_{b}, \\
\Gamma ( ( \rho ,\eta ) T\varphi _{H},\varphi _{H}) (
Y_{a}{\dot{{\tilde{\partial}}}}^{a}) &\!\!\!\!=(
Y_{a}H^{ab}) \circ \varphi _{L}{\dot{{\tilde{\partial}}}}_{b},%
\end{array}
\label{eq111}
\end{equation*}%
for any $Z^{\alpha }\overset{\ast }{\tilde{\partial}}_{\alpha }+Y_{a}\overset%
{\cdot }{\tilde{\partial}}^{a}\in \Gamma ( ( \rho ,\eta ) T%
\overset{\ast }{E},( \rho ,\eta ) \tau _{\overset{\ast }{E}},%
\overset{\ast }{E}) .$ The vector bundles morphism $\left( \left(
\rho ,\eta \right) T\varphi _{H},\varphi _{H}\right) $ will be
called the \emph{tangent }$\left( \rho ,\eta \right)
$-\emph{application of the Legendre bundles morphism associated to
the Hamiltonian }$H$. Using this application together with
 (\ref{L5}) and (\ref{L6}) we deduce the
following theorems.
\begin{theorem}
If $\overset{\ast }{u}=u_{a}s^{a}\in \Gamma ( \overset{\ast }{E},%
\overset{\ast }{\pi },M) $ such that
\begin{equation*}
\begin{array}{c}
\Gamma ( ( \rho ,\eta ) T\varphi _{H},\varphi _{H})
( \overset{\ast }{u}^{V}) =\Gamma(\varphi_H, Id_M)(\overset{\ast }{u}),%
\end{array}%
\end{equation*}%
then
\begin{align*}
u_{a}\circ \overset{\ast }{\pi }\circ \varphi _{L}=u_{a}\circ \pi, \ \ \ 
H^{ab}\circ \varphi _{L}=H^{ab}\circ \overset{\ast }{u}\circ \pi .%
\end{align*}
\end{theorem}
\begin{theorem}
If $\overset{\ast }{u}=u_{a}s^{a}\in \Gamma ( \overset{\ast }{E},%
\overset{\ast }{\pi },M)$ such that
\begin{equation*}
\begin{array}{c}
\Gamma ( ( \rho ,\eta ) T\varphi _{H},\varphi _{H})
( \overset{\ast }{u}^{C})=\Gamma(\varphi_H, Id_M)(\overset{\ast }{u}),%
\end{array}%
\end{equation*}%
then%
\begin{equation*}
\begin{array}{rl}
( g_{e}^{\alpha }u^{e}) \circ \pi &\!\!\!\!=( g^{\alpha
e}u_{e}) \circ \overset{\ast }{\pi }\circ \varphi _{L}, \\
U^{a}( K_{a}^{\gamma }( u) \circ h\circ \pi ) ( \tilde{g}_{\gamma
}^{b}\circ \pi ) &\!\!\!\!=\{ U_{a}( K^{\gamma a}( \overset{\ast
}{u}) \circ h\circ \overset{\ast }{\pi }) ( \tilde{g}_{c\gamma
}\circ \overset{\ast }{\pi }) H^{cb}\}
\circ \varphi _{L} \\
& -\{ ( \rho _{\alpha }^{i}\circ h\circ \overset{\ast }{\pi }%
) ( ( g_{e}^{\alpha }u^{e}) \circ \overset{\ast }{\pi }%
) H_{i}^{b}\} \circ \varphi _{L}.%
\end{array}%
\end{equation*}
\end{theorem}
\section{Duality between Lie algebroids structures}
\begin{theorem}
If the vector bundles morphism $\left( \left( \rho ,\eta \right)
T\varphi _{L},\varphi _{L}\right) $ is a morphism of Lie algebroids,
then we obtain:
\begin{equation*}
\begin{array}{c}
( L_{\alpha \beta }^{\gamma }\circ h\circ \pi ) \circ \varphi
_{H}=L_{\alpha \beta }^{\gamma }\circ h\circ \overset{\ast }{\pi },%
\end{array}
\label{eq112}
\end{equation*}%
\begin{equation*}
\begin{array}{cl}
&(( L_{\alpha \beta }^{\gamma }\rho _{\gamma }^{k}) \circ
h\circ \pi \cdot L_{kb})\circ \varphi _{H}=\rho _{\alpha }^{i}{%
\circ }h{\circ }\overset{\ast }{\pi }\cdot \frac{\partial }{\partial x^{i}}%
(( \rho _{\beta }^{j}{\circ }h{\circ }\pi \cdot L_{jb})
\circ \varphi _{H})\\
& -\rho _{\beta }^{j}{\circ }h{\circ }\overset{\ast }{\pi }\cdot \frac{%
\partial }{\partial x^{j}}(( \rho _{\alpha }^{i}{\circ }h{\circ }%
\pi \cdot L_{ib}) \circ \varphi _{H})\\
& +( \rho _{\alpha }^{i}{\circ }h{\circ }\pi \cdot L_{ia}) \circ
\varphi _{H}\cdot \frac{\partial }{\partial p_{a}}(( \rho _{\beta
}^{j}{\circ }h{\circ }\pi \cdot L_{jb}) \circ \varphi _{H})\\
& -( \rho _{\beta }^{j}{\circ }h{\circ }\pi \cdot L_{ja}) \circ
\varphi _{H}\cdot \frac{\partial }{\partial p_{a}}(( \rho
_{\alpha }^{i}{\circ }h{\circ }\pi \cdot L_{ib}) \circ \varphi _{H}%
),%
\end{array}
\label{eq113}
\end{equation*}%
\begin{equation*}
\begin{array}{cl}
0&\!\!\!\!=\rho _{\alpha }^{i}{\circ }h{\circ }\overset{\ast }{\pi }\cdot \frac{%
\partial }{\partial x^{i}}( L_{ba}\circ \varphi _{H}) +( \rho _{\alpha }^{i}{\circ }h{\circ }\pi \cdot L_{bc}) \circ
\varphi _{H}\frac{\partial }{\partial p_{c}}( L_{ia}\circ \varphi
_{H}) \\
&\ \ \ -L_{bc}\circ \varphi _{H}\cdot \frac{\partial }{\partial
p_{c}}(( \rho _{\alpha }^{i}{\circ }h{\circ }\pi \cdot L_{ia})
\circ
\varphi _{H}),
\end{array}
\label{eq114}
\end{equation*}%
and
\begin{equation*}
\begin{array}{cl}
0=L_{ac}\circ \varphi _{H}\cdot \frac{\partial }{\partial
p_{c}}\left( L_{bd}\circ \varphi _{H}\right)-L_{bd}\circ \varphi
_{H}\cdot \frac{\partial }{\partial p_{d}}\left(
L_{ac}\circ \varphi _{H}\right) .%
\end{array}
\label{eq115}
\end{equation*}
\end{theorem}

\begin{proof}
Developing the following equalities
\begin{align*}
&\Gamma ( ( \rho ,\eta ) T\varphi _{L},\varphi _{L}) [
\tilde{\partial}_{\alpha },\tilde{\partial}_{\beta }] _{(
\rho ,\eta ) TE} \\
&=[ \Gamma ( ( \rho ,\eta ) T\varphi _{L},\varphi _{L})
\tilde{\partial}_{\alpha },\Gamma ( ( \rho ,\eta ) T\varphi
_{L},\varphi _{L}) \tilde{\partial}_{\beta }]
_{( \rho ,\eta ) T\overset{\ast }{E}},%
\end{align*}%
\begin{align*}
&\Gamma ( ( \rho ,\eta ) T\varphi _{L},\varphi _{L})
[ \tilde{\partial}_{\alpha }, {\dot{{\tilde{\partial}}}}_{b}%
] _{( \rho ,\eta ) TE} \\
&=[ \Gamma ( ( \rho ,\eta ) T\varphi _{L},\varphi _{L})
\tilde{\partial}_{\alpha },\Gamma ( ( \rho ,\eta ) T\varphi
_{L},\varphi _{L}){\dot{{\tilde{\partial}}}}_{b}] _{( \rho ,\eta )
T\overset{\ast }{E}},
\end{align*}
and
\begin{align*}
&\Gamma ( ( \rho ,\eta ) T\varphi _{L},\varphi _{L})
[{\dot{{\tilde{\partial}}}}_{a}, {\dot{{\tilde{\partial}}}}_{b}] _{( \rho ,\eta ) TE} \\
&=[ \Gamma ( ( \rho ,\eta ) T\varphi _{L},\varphi _{L})
{\dot{{\tilde{\partial}}}}_{a},\Gamma ( ( \rho ,\eta ) T\varphi
_{L},\varphi _{L}){\dot{{\tilde{\partial}}}}_{b}] _{( \rho ,\eta )
T\overset{\ast }{E}},
\end{align*}
it results the conclusion of the theorem.
\end{proof}

\begin{corollary}
In the particular case of Lie algebroids, $\left( \eta
,h\right) =\left( Id_{M},Id_{M}\right) $, we obtain:
\begin{equation*}
\begin{array}{c}
( L_{\alpha \beta }^{\gamma }\circ \pi ) \circ \varphi
_{H}=L_{\alpha \beta }^{\gamma }\circ \overset{\ast }{\pi },%
\end{array}
\label{eq116}
\end{equation*}%
\begin{equation*}
\begin{array}{cl}
&(( L_{\alpha \beta }^{\gamma }\rho _{\gamma }^{k}) \circ
\pi \cdot L_{kb}) \circ \varphi _{H} =\rho _{\alpha }^{i}{\circ }%
\overset{\ast }{\pi }\cdot \frac{\partial }{\partial x^{i}}(( \rho
_{\beta }^{j}{\circ }\pi \cdot L_{jb}) \circ \varphi _{H})
\\
& -\rho _{\beta }^{j}{\circ }\overset{\ast }{\pi }\cdot \frac{\partial }{%
\partial x^{j}}(( \rho _{\alpha }^{i}{\circ }\pi \cdot
L_{ib}) \circ \varphi _{H})\\
& +( \rho _{\alpha }^{i}{\circ }\pi \cdot L_{ia}) \circ \varphi
_{H}\cdot \frac{\partial }{\partial p_{a}}(( \rho _{\beta }^{j}{%
\circ }\pi \cdot L_{jb}) \circ \varphi _{H}) \\
& -( \rho _{\beta }^{j}{\circ }\pi \cdot L_{ja}) \circ \varphi
_{H}\cdot \frac{\partial }{\partial p_{a}}(( \rho _{\alpha }^{i}{%
\circ }\pi \cdot L_{ib}) \circ \varphi _{H}),%
\end{array}
\label{eq117}
\end{equation*}%
\begin{equation*}
\begin{array}{cl}
0&\!\!\!\!=\rho _{\alpha }^{i}{\circ }\overset{\ast }{\pi }\cdot \frac{\partial }{%
\partial x^{i}}( L_{ba}\circ \varphi _{H})+( \rho _{\alpha }^{i}{\circ }\pi \cdot L_{bc}) \circ \varphi
_{H}\frac{\partial }{\partial p_{c}}( L_{ia}\circ \varphi _{H})
\\
& -L_{bc}\circ \varphi _{H}\cdot \frac{\partial }{\partial p_{c}}(
\left( \rho _{\alpha }^{i}{\circ }\pi \cdot L_{ia}\right) \circ \varphi _{H}%
),
\end{array}
\label{eq118}
\end{equation*}%
and
\begin{equation*}
\begin{array}{cl}
0&\!\!\!\!=L_{ac}\circ \varphi _{H}\cdot \frac{\partial }{\partial p_{c}}(
L_{bd}\circ \varphi _{H})-L_{bc}\circ \varphi _{H}\cdot
\frac{\partial }{\partial p_{c}}(
L_{ad}\circ \varphi _{H}) .%
\end{array}
\label{eq119}
\end{equation*}
\end{corollary}

\emph{In the classical case, }$\left( \rho ,\eta ,h\right) =\left(
Id_{TM},Id_{M},Id_{M}\right) $\emph{, we obtain:}%
\begin{equation*}
\begin{array}{cl}
0 &\!\!\!=\frac{\partial }{\partial x^{i}}( \frac{\partial
^{2}L}{\partial x^{j}\partial y^{k}}\circ \varphi _{H})
-\frac{\partial }{\partial x^{j}}( \frac{\partial ^{2}L}{\partial
x^{i}\partial y^{k}}\circ
\varphi _{H}) \\
& +\frac{\partial ^{2}L}{\partial x^{i}\partial y^{h}}\circ \varphi
_{H}\cdot \frac{\partial }{\partial p_{h}}( \frac{\partial ^{2}L}{%
\partial x^{j}\partial y^{k}}\circ \varphi _{H}) -\frac{\partial ^{2}L%
}{\partial x^{j}\partial y^{h}}\circ \varphi _{H}\cdot \frac{\partial }{%
\partial p_{h}}( \frac{\partial ^{2}L}{\partial x^{i}\partial y^{k}}%
\circ \varphi _{H}),
\end{array}
\label{eq120}
\end{equation*}%
\begin{equation*}
\begin{array}{cl}
0 &\!\!\!=\frac{\partial }{\partial x^{i}}( \frac{\partial
^{2}L}{\partial
y^{j}\partial y^{k}}\circ \varphi _{H}) +\frac{\partial ^{2}L}{%
\partial x^{i}\partial y^{h}}\circ \varphi _{H}\cdot \frac{\partial }{%
\partial p_{h}}( \frac{\partial ^{2}L}{\partial y^{j}\partial y^{k}}%
\circ \varphi _{H}) \\
& -\frac{\partial ^{2}L}{\partial x^{j}\partial y^{h}}\circ \varphi
_{H}\cdot \frac{\partial }{\partial p_{h}}( \frac{\partial ^{2}L}{%
\partial x^{i}\partial y^{k}}\circ \varphi _{H}),
\end{array}
\label{eq121}
\end{equation*}%
\emph{and}%
\begin{equation*}
\begin{array}{cl}
0 &\!\!\!=\frac{\partial ^{2}L}{\partial y^{i}\partial y^{h}}\circ
\varphi
_{H}\cdot \frac{\partial }{\partial p_{h}}( \frac{\partial ^{2}L}{%
\partial y^{j}\partial y^{k}}\circ \varphi _{H})-\frac{\partial ^{2}L}{\partial y^{j}\partial y^{h}}\circ \varphi
_{H}\cdot \frac{\partial }{\partial p_{h}}( \frac{\partial ^{2}L}{%
\partial y^{i}\partial y^{k}}\circ \varphi _{H}) .%
\end{array}
\label{eq122}
\end{equation*}
\begin{theorem}
If the vector bundles morphism $( ( \rho ,\eta ) T\varphi
_{H},\varphi _{H}) $ is a morphism of Lie algebroids, then we
obtain:
\begin{equation*}
\begin{array}{c}
( L_{\alpha \beta }^{\gamma }\circ h\circ \overset{\ast }{\pi })
\circ \varphi _{L}=L_{\alpha \beta }^{\gamma }\circ h\circ \pi ,%
\end{array}
\label{eq123}
\end{equation*}%
\begin{equation*}
\begin{array}{cl}
&(( L_{\alpha \beta }^{\gamma }\rho _{\gamma }^{k}) \circ h\circ
\overset{\ast }{\pi }\cdot H_{k}^{b}) \circ \varphi _{L}
=\rho _{\alpha }^{i}{\circ }h{\circ }\pi \cdot
\frac{\partial }{\partial
x^{i}}(( \rho _{\beta }^{j}{\circ }h{\circ }\overset{\ast }{\pi }%
\cdot H_{j}^{b}) \circ \varphi _{L})\\
& -\rho _{\beta }^{j}{\circ }h{\circ }\pi \cdot \frac{\partial }{\partial
x^{j}}(( \rho _{\alpha }^{i}{\circ }h{\circ }\overset{\ast }{\pi }%
\cdot H_{i}^{b}) \circ \varphi _{L})\\
& +( \rho _{\alpha }^{i}{\circ }h{\circ }\overset{\ast }{\pi }\cdot
H_{i}^{c}) \circ \varphi _{L}\cdot \frac{\partial }{\partial y^{c}}%
(( \rho _{\beta }^{j}{\circ }h{\circ }\overset{\ast }{\pi }\cdot
H_{j}^{b}) \circ \varphi _{L})\\
& -( \rho _{\beta }^{j}{\circ }h{\circ }\overset{\ast }{\pi }\cdot
H_{j}^{c}) \circ \varphi _{L}\cdot \frac{\partial }{\partial y^{c}}%
(( \rho _{\alpha }^{i}{\circ }h{\circ }\overset{\ast }{\pi }\cdot
H_{i}^{b}) \circ \varphi _{L}),%
\end{array}
\label{eq124}
\end{equation*}%
\begin{equation*}
\begin{array}{cl}
0 &\!\!\!\!=\rho _{\alpha }^{i}{\circ }h{\circ }\pi \cdot
\frac{\partial }{\partial x^{i}}( H^{ba}\circ \varphi _{L})+( \rho
_{\alpha }^{i}{\circ }h{\circ }\overset{\ast }{\pi }\cdot H^{bc})
\circ \varphi _{L}\frac{\partial }{\partial y^{c}}(
H^{ba}\circ \varphi _{L}) \\
& -H^{bc}\circ \varphi _{L}\cdot \frac{\partial }{\partial y^{c}}((
\rho _{\alpha }^{i}{\circ }h{\circ }\overset{\ast }{\pi }\cdot
H_{i}^{a}) \circ \varphi _{L}),
\end{array}
\label{eq125}
\end{equation*}%
and
\begin{equation*}
\begin{array}{cl}
0 &\!\!\!\!=H^{ac}\circ \varphi _{L}\cdot \frac{\partial }{\partial
y^{c}}( H^{bd}\circ \varphi _{L}) -H^{bc}\circ \varphi _{L}\cdot
\frac{\partial }{\partial y^{c}}(
H^{ad}\circ \varphi _{L}) .%
\end{array}
\label{eq126}
\end{equation*}
\end{theorem}
\begin{proof}
Developing the equalities
\begin{equation*}
\begin{array}{c}
\Gamma ( ( \rho ,\eta ) T\varphi _{H},\varphi _{H})
[ \overset{\ast }{\tilde{\partial}}_{\alpha },\overset{\ast }{\tilde{%
\partial}}_{\beta }] _{( \rho ,\eta ) T\overset{\ast }{E}}
\\
=[ \Gamma ( ( \rho ,\eta ) T\varphi _{H},\varphi _{H}) \overset{\ast
}{\tilde{\partial}}_{\alpha },\Gamma ( (
\rho ,\eta ) T\varphi _{H},\varphi _{H}) \overset{\ast }{\tilde{%
\partial}}_{\beta }] _{( \rho ,\eta ) TE},%
\end{array}%
\end{equation*}%
\begin{equation*}
\begin{array}{c}
\Gamma ( ( \rho ,\eta ) T\varphi _{H},\varphi _{H})
[ \overset{\ast }{\tilde{\partial}}_{\alpha },{\dot{{\tilde{%
\partial}}}}^{b}] _{( \rho ,\eta ) T\overset{\ast }{E}} \\
=[ \Gamma ( ( \rho ,\eta ) T\varphi _{H},\varphi _{H}) \overset{\ast
}{\tilde{\partial}}_{\alpha },\Gamma ( (
\rho ,\eta ) T\varphi _{H},\varphi _{H}){\dot{\tilde{%
\partial}}}^{b}] _{( \rho ,\eta ) TE},%
\end{array}%
\end{equation*}%
and
\begin{equation*}
\begin{array}{c}
\Gamma ( ( \rho ,\eta ) T\varphi _{H},\varphi _{H})
[{\dot{{\tilde{\partial}}}}^{a},{\dot{{\tilde{%
\partial}}}}^{b}] _{( \rho ,\eta ) T\overset{\ast }{E}} \\
=[ \Gamma ( ( \rho ,\eta ) T\varphi _{H},\varphi _{H})
{\dot{{\tilde{\partial}}}}^{a},\Gamma ( ( \rho
,\eta ) T\varphi _{H},\varphi _{H}){\dot{{\tilde{%
\partial}}}}^{b}] _{( \rho ,\eta ) TE},%
\end{array}%
\end{equation*}%
it results the conclusion of the theorem.
\end{proof}

\begin{corollary}
\label{c82}\textbf{\ } In the particular case of Lie algebroids, $%
( \eta ,h) =( Id_{M},Id_{M}) $, we obtain%
\begin{equation*}
\begin{array}{c}
( L_{\alpha \beta }^{\gamma }\circ \overset{\ast }{\pi }) \circ
\varphi _{L}\!=L_{\alpha \beta }^{\gamma }\circ \pi ,%
\end{array}
\label{eq127}
\end{equation*}%
\begin{equation*}
\begin{array}{cl}
&(( L_{\alpha \beta }^{\gamma }\rho _{\gamma }^{k}) \circ
\overset{\ast }{\pi }\cdot H_{k}^{b}) \circ \varphi _{L}=\rho
_{\alpha }^{i}{\circ }\pi \cdot \frac{\partial }{\partial x^{i}}((
\rho _{\beta }^{j}{\circ }\overset{\ast }{\pi }\cdot H_{j}^{b})
\circ \varphi _{L}) \\
& -\rho _{\beta }^{j}{\circ }\pi \cdot \frac{\partial }{\partial
x^{j}}( ( \rho _{\alpha }^{i}{\circ }\overset{\ast }{\pi }\cdot
H_{i}^{b}) \circ \varphi _{L})+( \rho _{\alpha }^{i}{\circ
}\overset{\ast }{\pi }\cdot
H_{i}^{c}) \circ \varphi _{L}\cdot \frac{\partial }{\partial y^{c}}%
(( \rho _{\beta }^{j}{\circ }\overset{\ast }{\pi }\cdot
H_{j}^{b}) \circ \varphi _{L})\\
& -( \rho _{\beta }^{j}{\circ }\overset{\ast }{\pi }\cdot
H_{j}^{c}) \circ \varphi _{L}\cdot \frac{\partial }{\partial y^{c}}%
(( \rho _{\alpha }^{i}{\circ }\overset{\ast }{\pi }\cdot
H_{i}^{b}) \circ \varphi _{L}),%
\end{array}
\label{eq128}
\end{equation*}%
\begin{equation*}
\begin{array}{cl}
0 &\!\!\!\!=\rho _{\alpha }^{i}{\circ }\pi \cdot \frac{\partial }{\partial x^{i}}%
( H^{ba}\circ \varphi _{L})+( \rho _{\alpha }^{i}{\circ
}\overset{\ast }{\pi }\cdot H^{bc}) \circ \varphi _{L}\frac{\partial
}{\partial y^{c}}(
H^{ba}\circ \varphi _{L}) \\
& -H^{bc}\circ \varphi _{L}\cdot \frac{\partial }{\partial y^{c}}((
\rho _{\alpha }^{i}{\circ }\overset{\ast }{\pi }\cdot
H_{i}^{a}) \circ \varphi _{L}),
\end{array}
\label{eq129}
\end{equation*}%
and
\begin{equation*}
\begin{array}{cl}
0 &\!\!\!\!=H^{ac}\circ \varphi _{L}\cdot \frac{\partial }{\partial
y^{c}}( H^{bd}\circ \varphi _{L})-H^{bc}\circ \varphi _{L}\cdot
\frac{\partial }{\partial y^{c}}(
H^{ad}\circ \varphi _{L}) .%
\end{array}
\label{eq130}
\end{equation*}
\end{corollary}
\emph{In the classical case, }$( \rho ,\eta ,h) =(
Id_{TM},Id_{M},Id_{M}) $\emph{, we obtain}%
\begin{equation*}
\begin{array}{cl}
0 & =\frac{\partial }{\partial x^{i}}( \frac{\partial
^{2}H}{\partial x^{k}\partial p_{j}}\circ \varphi _{L})
-\frac{\partial }{\partial x^{k}}( \frac{\partial ^{2}H}{\partial
x^{i}\partial p_{j}}\circ \varphi _{L}) +\frac{\partial
^{2}H}{\partial x^{i}\partial p_{h}}\circ \varphi
_{L}\cdot \frac{\partial }{\partial y^{h}}( \frac{\partial ^{2}H}{%
\partial x^{k}\partial p_{j}}\circ \varphi _{L}) \\
& -\frac{\partial ^{2}H}{\partial x^{k}\partial p_{h}}\circ \varphi
_{L}\cdot \frac{\partial }{\partial y^{h}}( ( \frac{\partial ^{2}H%
}{\partial x^{i}\partial p_{j}}\circ \varphi _{L}) \circ \varphi
_{L}),
\end{array}
\label{eq131}
\end{equation*}%
\begin{equation*}
\begin{array}{cl}
0 &\!\!\!\!=\frac{\partial }{\partial x^{k}}( \frac{\partial
^{2}H}{\partial
p_{i}\partial p_{j}}\circ \varphi _{L}) +\frac{\partial ^{2}H}{%
\partial p_{i}\partial p_{h}}\circ \varphi _{L}\cdot \frac{\partial }{%
\partial y^{h}}( \frac{\partial ^{2}H}{\partial x_{k}\partial p_{j}}%
\circ \varphi _{L}) \\
& -\frac{\partial ^{2}H}{\partial p_{j}\partial p_{h}}\circ \varphi
_{L}\cdot \frac{\partial }{\partial y^{h}}( \frac{\partial ^{2}H}{%
\partial x^{k}\partial p_{i}}\circ \varphi _{L}),
\end{array}
\label{eq132}
\end{equation*}%
\emph{and}%
\begin{equation*}
\begin{array}{cl}
0 &\!\!\!\!=\frac{\partial ^{2}H}{\partial p_{i}\partial p_{k}}\circ
\varphi
_{L}\cdot \frac{\partial }{\partial y^{k}}( \frac{\partial ^{2}H}{%
\partial p_{j}\partial p_{h}}\circ \varphi _{L})-\frac{\partial ^{2}H}{\partial p_{j}\partial p_{k}}\circ \varphi
_{L}\cdot \frac{\partial }{\partial y^{k}}( \frac{\partial ^{2}H}{%
\partial p_{i}\partial p_{h}}\circ \varphi _{L}) .%
\end{array}
\label{eq133}
\end{equation*}
\begin{definition}
If $( ( \rho ,\eta ) T\varphi _{L},\varphi _{L}) $ and $( ( \rho
,\eta ) T\varphi _{H},\varphi _{H}) $ are
Lie algebroids morphisms, then we will say that $( E,\pi ,M) $%
and $( \overset{\ast }{E},\overset{\ast }{\pi },M) $  are Legendre
$( \rho ,\eta) $-equivalent and we will write
\begin{equation*}
\begin{array}{c}
( E,\pi ,M) \overset{\mathcal{L}}{\widetilde{_{( \rho ,\eta) }}}( \overset{\ast }{E},\overset{\ast }{\pi },M) .%
\end{array}%
\end{equation*}
\end{definition}
{\bf\Large Acknowledgment}\\\\
The second author would like to give his warmest thanks to Radinesti-Gorj Cultural Sciantifique Society for financial support.


\noindent
Esmail Peyghan and Leila Nourmohammadifar\\
Department of Mathematics, Faculty  of Science\\
Arak University\\
Arak 38156-8-8349,  Iran\\
Email: e-peyghan@araku.ac.ir,\ \ l.nourmohammadi@gmail.com\\

\noindent
Constantin M Arcu\c{s}\\
Secondary School "Cornelius Radu"\\
Radinesti Village, 217196\\
Gorj County, Romania\\
Email:\ c\_arcus@radinesti.ro

\end{document}